\documentclass{amsart}
\usepackage{amssymb}
\usepackage{eucal}
\usepackage{amsfonts}
\usepackage{epsfig}
\usepackage[all]{xy}

\let\oldcite\cite                                  

\newtheorem{thm}{Theorem}[section]
\newtheorem{cor}[thm]{Corollary}
\newtheorem{lem}[thm]{Lemma}
\newtheorem{prop}[thm]{Proposition}
\theoremstyle{definition}
\newtheorem{defn}[thm]{Definition}
\theoremstyle{remark}
\newtheorem{rem}[thm]{Remark}
\numberwithin{equation}{section} \theoremstyle{remark}
\newtheorem{ex}[thm]{Example}

\newcommand{\mcY}{\mathcal{Y}}
\newcommand{\mcZ}{\mathcal{Z}}
\newcommand{\mcK}{\mathcal{K}}
\newcommand{\mcG}{\mathcal{G}}
\newcommand{\mcH}{\mathcal{H}}
\newcommand{\mcP}{\mathcal{P}}

\newcommand{\mfZ}{\mathfrak{Z}}

\newcommand{\bbG}{\mathbb{G}}
\newcommand{\bbH}{\mathbb{H}}
\newcommand{\bbK}{\mathbb{K}}

\newcommand{\mcGg}{\mathcal{G}_{\Gamma}}
\newcommand{\bbGg}{\mathbb{G}_{\Gamma}}

\newcommand{\al}{\alpha}
\newcommand{\be}{\beta}
\newcommand{\vep}{\varepsilon}
\newcommand{\si}{\sigma}
\newcommand{\ups}{\upsilon}
\newcommand{\vart}{\vartheta}
\newcommand{\vf}{\varphi}

\newcommand{\urho}{\rho}

\let\:=\colon

\newcommand{\Ob}{\operatorname{Ob}}

\newcommand{\id}{\operatorname{id}}
\newcommand{\pr}{\operatorname{pr}}

\newcommand{\coker}{\operatorname{coker}}

\def\smashedst{\setbox0=\hbox{$\rightrightarrows$}\ht0=4pt\box0}
\newcommand{\sst}[1]{\stackrel{#1}{\smashedst}}

\newcommand{\lra}{\longrightarrow}
\newcommand{\llra}[1]{\stackrel{#1}{\lra}}

\newcommand{\oux}[2]{\underset{#1}{\overset{#2}\times}}
\newcommand{\ox}[1]{\overset{#1}{\times}}

\begin{document}

\title{Group cohomology with coefficients in a crossed-module}

\author{Behrang Noohi}

\begin{abstract}
We compare three different ways of defining group cohomology with
coefficients in a crossed-module: 1) explicit approach via cocycles; 2)
geometric approach via gerbes; 3) group theoretic approach via butterflies. We
discuss the case where the crossed-module is braided and the case where the
braiding is symmetric. We prove the functoriality of the cohomologies with
respect to weak morphisms of crossed-modules and also prove the ``long'' exact
cohomology sequence associated to a short exact sequence of crossed-modules and
weak morphisms.
\end{abstract}

\maketitle%

\tableofcontents%

\section{Introduction}

These notes grew out as an attempt to answer certain questions
regarding group cohomology with coefficients in a crossed-module
which were posed to us by M.~Borovoi in relation to his work on
abelian Galois cohomology of reductive groups \cite{Bor}.

We collect some known or not so-well-known results in this area and
put them in a coherent (and hopefully user-friendly) form, as well
as add our own new approach to the subject via butterflies.
We hope that the application-minded user finds these notes beneficial.
We especially expect these result to be useful in Galois cohomology (e.g., in 
the study of relative Picard groups of Brauer groups).

Let us outline the content of the paper. Let $\Gamma$ be a group
acting strictly on a crossed-module $\bbG$. We investigate the group
cohomologies $H^{i}(\Gamma,\bbG)$. We compare three different ways
of constructing the cohomologies $H^{i}(\Gamma,\bbG)$, $i=-1,0,1$.
One approach is entirely new. We also work out some novel aspects of
the other two approaches which, to our knowledge, were not
considered previously.

Let us briefly describe the three approaches that we are
considering. 

\medskip
\noindent{\em The cocycle approach.}
The first approach uses an explicit cocycle description
of the cohomology groups.  Many people have worked on this.
The original idea goes back to Dedecker
\cite{Dedecker1,Dedecker2}. But he only considers the case of
a trivial $\Gamma$-action (where things
get oversimplified). Borovoi \cite{Bor} treats the general case in
his study of abelianization  of Galois cohomology of reductive
groups. A systematic approach is developed in \cite{CeFe} where
the more general case of a 2-group fibered over a category is treated.
The cohomology groups with coefficients in a symmetric
braided crossed-module have been studied in
\cite{CaMa,BuCaCe,Ul}. The paper of Garz\'on and del R\'io \cite{GadR} 
seems to be
the first place where the group structure on $H^1$ appears in print.
Breen's letter to Borovoi \cite{BreenBor} also discusses the group structure 
on $H^1$ in the crossed-module language and gives explicit formulas.
Also relevant is the work \cite{BGPT} in
which the authors study homotopy types of equivariant
crossed-complexes. 

We also point out that there is a standard way of going from \u Cech
cohomology to group cohomology, as discussed in (\cite{Tannakian},
$\S$ 5.7). In this way, it is possible in principle to deduce results 
about group cohomology from Breen's general 
results on \u Cech cohomology.

In Sections \ref{S:0}-\ref{S:1} we rework the definitions of $H^i$, $i=-1,0,1$.
The only originality we may claim in these sections
is merely in the form of presentation (e.g., explicitly working out all the formulas
in the language of crossed-modules), as the concepts are well understood. 

What seems to  be original here is that in $\S$\ref{S:B} we introduce an explicit 
crossed-module in groupoids concentrated in degrees $[-1,1]$, denoted 
$\mcK^{\leq 1}(\Gamma,\bbG)$, whose cohomologies are precisely
$H^{i}(\Gamma,\bbG)$. This crossed-module in groupoids encodes everything 
that is known about the $H^i$ (and more).

In the case where $\bbG$ is endowed with a
$\Gamma$-equivariant braiding, we show that $\mcK^{\leq
1}(\Gamma,\bbG)$ is a 2-crossed-module. In particular,
$H^{1}(\Gamma,\bbG)$ is a group and $H^{i}(\Gamma,\bbG)$ are
abelian, $i=-1,0$. When the braiding is symmetric, we show  that
$\mcK^{\leq 1}(\Gamma,\bbG)$ is a braided 2-crossed-module. This implies that
$H^{1}(\Gamma,\bbG)$ is also abelian. 

We also prove that $\mcK^{\leq 1}(\Gamma,\bbG)$ is functorial in
(strict) $\Gamma$-equivariant morphisms of crossed-modules and takes
an equivalence of crossed-modules to an equivalence of
crossed-modules in groupoids (or of 2-crossed-modules, 
resp., braided 2-crossed-modules, 
in the case where $\bbG$ is braided, resp., symmetric).
In particular, an equivalence of crossed-modules induces
an isomorphism on all $H^i$.

\medskip
\noindent{\em The butterfly approach.}
In the second approach, we construct a 2-groupoid
$\mfZ(\Gamma,\bbG)$ such that $H^{i}(\Gamma,\bbG)\cong
\pi_{1-i}\mfZ(\Gamma,\bbG)$. The objects of this 2-groupoid are
certain diagrams of groups involving $\Gamma$ and the $G_i$; see
$\S$\ref{SS:H1new}. In $\S$\ref{SS:relation} we give an sketch of
how to construct a biequivalence between $\mcK^{\leq
1}(\Gamma,\bbG)$ and the crossed-module in groupoids associated to
$\mfZ(\Gamma,\bbG)$. 

We show that $\mfZ(\Gamma,\bbG)$ is  functorial
in weak $\Gamma$-equivariant morphisms $\bbH \to \bbG$ of
crossed-modules (read {\em strong $\Gamma$-equivariant
butterflies}). In particular, it takes an equivalence of butterflies
to an equivalence of 2-groupoids. In the case where $\bbG$ is
endowed  with a $\Gamma$-equivariant braiding, we endow
$\mfZ(\Gamma,\bbG)$ with a natural monoidal structure which makes it
a group object in the category of 2-groupoids and weak functors.
Under the equivalence between $\mfZ(\Gamma,\bbG)$ and $\mcK^{\leq
1}(\Gamma,\bbG)$, the group structure on the former corresponds to
the 2-crossed-module structure on the latter. In the case where the
braiding on $\bbG$ is symmetric,
$\mfZ(\Gamma,\bbG)$ admits a symmetric braiding.

\medskip
\noindent{\em The gerbe approach.}
Finally, the third approach is that of Breen \cite{Bitors} adopted
to our specific situation (it is also closely related to \cite{CeFe}).
We  construct another 2-groupoid $\mfZ(\Gamma,\mcG)$ which we show is
naturally biequivalent to  $\mfZ(\Gamma,\bbG)$. Here, $\mcG$ is the
2-group associated to $\bbG$. The objects of this 2-groupoid are
principal $\mcG$-bundles over the classifying stack $B\Gamma$ of
$\Gamma$. We show that $\mfZ(\Gamma,\mcG)$ is  functorial in weak
$\Gamma$-equivariant morphisms $\bbH \to \bbG$ of crossed-modules
(read $\Gamma$-equivariant butterflies).

 In the case where $\bbG$ is
endowed  with a $\Gamma$-equivariant braiding, $\mfZ(\Gamma,\mcG)$
admits a natural monoidal structure which makes it a group object in
the category of 2-groupoids and weak functors. In this case,  the
equivalence between $\mfZ(\Gamma,\mcG)$ and $\mfZ(\Gamma,\bbG)$ is
monoidal. When  the braiding on $\bbG$ is symmetric,
$\mfZ(\Gamma,\mcG)$ admits a symmetric braiding and so does the
equivalence between $\mfZ(\Gamma,\mcG)$ and $\mfZ(\Gamma,\bbG)$.

\medskip

We also consider the last two approaches in the case where
everything is over a Grothendieck site. This is useful for geometric
applications in which $\Gamma$ and $\bbG$ are topological,
Lie, algebraic, and so on.

Finally, we show that a short exact sequence
        $$1 \to \bbK \to \bbH \to \bbG \to 1$$
of $\Gamma$-crossed-modules and weakly $\Gamma$-equivariant weak
morphisms (read $\Gamma$-butterflies) over a Grothendieck site gives
rise to a long exact cohomology sequence
    {\small $$ \xymatrix@C=12pt{%
      1 \ar[r] & H^{-1} (\Gamma,\bbK) \ar[r] & H^{-1} (\Gamma,\bbH)
        \ar[r] & H^{-1} (\Gamma,\bbG) \ar[r] &
      H^{0} (\Gamma,\bbK) \ar `[d] `[lll] `[dlll] [dlll] \\
      & H^{0} (\Gamma,\bbH) \ar[r] & H^{0} (\Gamma,\bbG) \ar[r] &
      H^{1} (\Gamma,\bbK)
      \ar[r] & H^{1} (\Gamma,\bbH) \ar[r] & H^{1} (\Gamma,\bbG).
                }   $$  }
Also see \cite{CeGa} and \cite{CeFe}, Theorem 31.

\medskip
\noindent{\em One last comment.} We end this introduction by
pointing out one serious omission in this paper: $H^2$. In the case
where $\bbG$ has a $\Gamma$-equivariant braiding, one expects to be
able to push the theory one step further to include $H^2$. In the
first approach, the complex $\mcK^{\leq 1}(\Gamma,\bbG)$ is expected
to be the $\leq 1$ truncation of a certain complex $\mcK^{\leq
2}(\Gamma,\bbG)$ concentrated in degrees $[-1,2]$. In the second and
the third approaches, the 2-groupoids  $\mfZ(\Gamma,\bbG)$ and
$\mfZ(\Gamma,\mcG)$ get  replaced by certain pointed 3-groupoids
whose automorphism 2-groupoids of the base object are
$\mfZ(\Gamma,\bbG)$ and  $\mfZ(\Gamma,\mcG)$. The long exact
cohomology sequence should also extend to include the three
additional $H^2$ terms.  The machinery for doing all this is being
developed in a forthcoming paper \cite{ButterflyIII} and is not
available yet in print. For that reason, we will not get into the
discussion of $H^2$ in these notes.

All of the above can also be done with the group $\Gamma$ replaced by a 
crossed-module (the action on $\bbG$ remains strict). This is useful 
because in some applications (e.g., in Galois cohomology) the action of
$\Gamma$ on $\bbG$ is not strict but it can be replaced with
a strict action of a 2-group equivalent to $\Gamma$. We will not pursue this topic
here and leave it to a future paper.

\medskip
\noindent{\bf Acknowledgement.} I am grateful to M.~Borovoi for
getting me interested in this work and for many useful email
conversations. I thank E.~Aldrovandi for a discussion we had
regarding this problem (especially 
$\S$\ref{SS:braidinggerbes}) and L.~Breen for his comments
on an earlier version of the paper. I would also like to thank Fernando
Muro who made me aware of the work of the Granada school. And a special
thanks to the refree for carefully reading the paper and making helpful suggestions. 
The content of Appendix A borrows greatly
from the referee's report.

\section{Notation and conventions}

Let $\bbG=[\partial\: G_1\to G_0]$ be  a crossed-module. We assume
that $G_0$ acts on $G_1$ on the right. We denote the action of $g
\in G_0$ on $\al \in G_1$ by $\al^g$. Let $\Gamma$ be a group acting
on a crossed-module $\bbG=[\partial\: G_1\to G_0]$ on the left.  We
denote the action of $\si \in \Gamma$ on an element $g$ by
${}^{\si}\!g$. We require the $\Gamma$ action on $\bbG$  to be
compatible with the action of $G_0$ on $G_1$ in the following way:
   $${}^{\si}\!(\al^g)=({}^{\si}\!\al)^{{}^{\si}\!g}.$$
We usually denote  $({}^{\si}\!\al)^g$ by
${}^{\si}\!\al^{g}$. Note that this is {\em not} equal to
${}^{\si}\!(\al^{g})$.

We refer to a crossed-module $\bbG$
equipped with an action by a group $\Gamma$ as  above as a {\em $\Gamma$-equivariant 
crossed-module} or for short as a {\em $\Gamma$-crossed-module}.

Our convention for  braiding $\{,\} \: G_0\times G_0 \to G_1$ is
that $\partial\{g,h\}=g^{-1}h^{-1}gh$. The braidings are assumed to
be $\Gamma$-equivariant in the sense that
$\{{}^{\si}\!g,{}^{\si}\!h\}={}^{\si}\!\{g,h\}$, for every
$g,h \in G_0$ and $\si \in \Gamma$.

Whenever there is fear of confusion, we use a dot $\cdot$ for
products in complicated formulas; the same products may appear
without a dot in other places (even in the same formulas).

All groupoids, 2-groupoids and so on are assumed to be small.

\section{$H^{-1}$ and $H^0$ of a $\Gamma$-equivariant crossed-module}
{\label{S:0}}

By definition, $H^{-1}(\Gamma,\bbG)=(\ker\partial)^{\Gamma}$. This
is an abelian group. Let us now define $H^{0}(\Gamma,\bbG)$.

A {\bf $0$-cochain} is a pair $(g, \theta)$ where $g \in G_0$ and
$\theta \: \Gamma \to G_1$ is a pointed map. We denote the set of
$0$-cochains by $C^{0}(\Gamma,\bbG)$. There is a multiplication on
$C^{0}(\Gamma,\bbG)$ which makes it into a group. By definition, the
product of two $0$-cochains $(g_1,\theta_1)$ and $(g_2,\theta_2)$ is
    $$(g_1,\theta_1)(g_2,\theta_2):=(g_1g_2,\theta_1^{g_2}\theta_2),$$
where $\theta_1^{g_2}\theta_2 \: \Gamma \to G_1$ is defined by $\si
\mapsto \theta_1(\si)^{g_2}\theta_2(\si)$.

\begin{rem}{\label{R:wrongprod}}
  In the case where $\bbG$ is braided, there is another way of
  making $C^{0}(\Gamma,\bbG)$ into a group. This will be discussed in
  $\S$\ref{SS:braiding0} and used later on in $\S$\ref{S:1}.
\end{rem}

A  0-cochain $(g, \theta)$ is a {\bf $0$-cocycle} if the following
conditions are satisfied:
 \begin{itemize}
    \item For every $\si \in \Gamma$,
        $\partial\theta(\si)=g^{-1}\cdot{}^{\si}\!g$,
    \item For every $\si,\tau \in \Gamma$, $\theta(\si\tau)=
         \theta(\si)\cdot{}^{\si}\!\theta(\tau)$.
 \end{itemize}
The 0-cocycles form a subgroup of  $C^{0}(\Gamma,\bbG)$ which  we denote
by $Z^{0}(\Gamma,\bbG)$.

An element in $Z^0(\Gamma,\bbG)$ is a {\bf $0$-coboundary} if it is
of the form $(\partial\mu,\theta_{\mu})$, where $\mu \in G_1$ and
$\theta_{\mu}  \: \Gamma \to G_1$ is defined by
$\theta_{\mu}(\si):=\mu^{-1}\cdot{}^{\si}\!\mu$. It is easy to see
that the set $B^{0}(\Gamma,\bbG)$ of 0-coboundaries is a normal
subgroup of $Z^{0}(\Gamma,\bbG)$; it is in fact normal in
$C^{0}(\Gamma,\bbG)$ too. We define
     $$H^{0}(\Gamma,\bbG):=\frac{Z^{0}(\Gamma,\bbG)}{B^{0}(\Gamma,\bbG)}.$$
This group is not in general abelian.

A better way of phrasing the above discussion is to say that
      $$[G_1 \to Z^0(\Gamma,\bbG)]$$
         $$\mu \mapsto (\partial\mu,\theta_{\mu})$$
is a crossed-module. The action of  $Z^0(\Gamma,\bbG)$ on $G_1$ is
defined by
    $$\mu^{(g,\theta)}:=\mu^g.$$
\subsection{In the presence of a braiding on $\bbG$}{\label{SS:braiding0}}

When $\bbG$ is braided, $H^{0}(\Gamma,\bbG)$ is abelian.
This is true thanks to the following.

\begin{lem}
  The commutator of the two $0$-cocycles $(g,\theta)$ and
  $(g',\theta')$ in $Z^{0}(\Gamma,\bbG)$ is equal to the
  $0$-coboundary $(\partial\mu,\theta_{\mu})$, where
  $\mu=\{g,g'\}$.
\end{lem}

In fact, it follows from the above lemma that the bracket
    $$\{(g,\theta),(g',\theta')\}:=\{g,g'\}$$
makes the crossed-module $[G_1 \to Z^0(\Gamma,\bbG)]$
defined at the end of the previous subsection into a braided crossed-module.

As we pointed out in Remark \ref{R:wrongprod}, in the presence of a
braiding on $\bbG$, there is a second product on
$C^{0}(\Gamma,\bbG)$ which makes it into a group as well. This new
product will be used in an essential way in $\S$\ref{S:1}. Here is
how it is defined. Given two 0-cochains $(g_1,\theta_1)$ and
$(g_2,\theta_2)$, their product is the 0-cochain $(g_1g_2,\vart)$,
where $\vart$ is defined by the formula
    $$
      \vart(\si):=\theta_1(\si)^{p_2(\si)}\cdot\theta_2(\si)
          \cdot\{g_2^{-1},g_1^{-1}{}^{\si}\!g_1\}^{{}^{\si}\!g_2}
     $$
Here, $p_2(\si)=g_2^{-1}\cdot{}^{\si}\!g_2\cdot\partial\theta_2(\si)^{-1}$.

It is not hard to check that, when restricted to
$Z^{0}(\Gamma,\bbG)$, the above product
coincides with the one defined in the previous subsection.

\section{$H^1$ of a $\Gamma$-equivariant crossed-module}{\label{S:1}}

A $1$-cocycle on $\Gamma$ with values in $\bbG$ is a pair $(p,\vep)$
where
   $$p\: \Gamma \to G_0 \ \ \text{and} \ \ \vep \: \Gamma \times \Gamma \to G_1$$
are pointed set maps satisfying the following conditions:
 \begin{itemize}
     \item For every
        $\si, \tau \in \Gamma$, $p(\si\tau)\cdot\partial\vep(\si,\tau)=
            p(\si)\cdot{}^{\si}\!p(\tau)$.
     \item For every $\si, \tau, \ups \in \Gamma$,
          $\vep(\si,\tau\ups)\cdot{}^{\si}\!\vep(\tau,\ups)=
           \vep(\si\tau,\ups)\cdot\vep(\si,\tau)^{{}^{\si\tau}\!p(\ups)}$.
 \end{itemize}

We denote the set of 1-cocycles by $Z^{1}(\Gamma,\bbG)$. This is a
pointed set with the base point being the pair of constant functions
$(1_{G_0},1_{G_1})$. In fact, $Z^{1}(\Gamma,\bbG)$
is the set of objects a groupoid $\mcZ^{1}(\Gamma,\bbG)$. An arrow
             $$(p_1,\vep_1) \to (p_2,\vep_2)$$
in $\mcZ^{1}(\Gamma,\bbG)$ is given by a pair $(g,\theta)$,
with $g \in G_0$ and $\theta \: \Gamma \to G_1$ a pointed map, such that
   \begin{itemize}
     \item for every $\si \in \Gamma$,
            $$p_2(\si)=g^{-1}\cdot p_1(\si)\cdot{}^{\si}\!g\cdot\partial\theta(\si)^{-1};$$
     \item for every $\si, \tau \in \Gamma$, 
     $$\vep_2(\si,\tau)=
        \theta(\si\tau)\cdot\vep_1(\si,\tau)^{{}^{\si\tau}\!g}\cdot{}^{\si}\!\theta(\tau)^{-1}
             \cdot\big(\theta(\si)^{-1}\big)^{{}^{\si}\!p_2(\tau)}.$$
   \end{itemize}

The above formulas can be interpreted as a right action of the group
$C^{0}(\Gamma,\bbG)$ (with the group structure introduced at the
beginning of $\S$\ref{S:0}) on the set $Z^{1}(\Gamma,\bbG)$. We
denote this right action by
   $$(p_2,\vep_2)=(p_1,\vep_1)^{(g,\theta)}.$$
The groupoid $\mcZ^{1}(\Gamma,\bbG)$ is
simply the transformation groupoid of this action.

We define $H^{1}(\Gamma,\bbG)$ to be the pointed set of isomorphism
classes of the groupoid $\mcZ^{1}(\Gamma,\bbG)$.

\subsection{In the presence of a braiding on $\bbG$}{\label{SS:braiding1}}

In the previous subsection, we constructed the groupoid
$\mcZ^{1}(\Gamma,\bbG)$ of 1-cocycles as the transformation groupoid
of a certain action of $C^{0}(\Gamma,\bbG)$ on $Z^{1}(\Gamma,\bbG)$.
In the case where $\bbG$ is endowed with a $\Gamma$-equivariant
braiding, we will see below that the set $Z^{1}(\Gamma,\bbG)$ itself
also has a group structure. In this situation, it is natural to ask
is whether the action of $C^{0}(\Gamma,\bbG)$ on
$Z^{1}(\Gamma,\bbG)$ is a (right) multiplication action via a
certain group homomorphism $C^{0}(\Gamma,\bbG) \to
Z^{1}(\Gamma,\bbG)$.

The answer to this question appears to be negative. However, if we
use the alternative group structure on $C^{0}(\Gamma,\bbG)$ that we
introduced in $\S$\ref{SS:braiding0}, then there does exist such a
group homomorphism $d \: C^{0}(\Gamma,\bbG) \to Z^{1}(\Gamma,\bbG)$.
We point out that the right multiplication action of
$C^{0}(\Gamma,\bbG)$ on $Z^{1}(\Gamma,\bbG)$ obtained via $d$ is
{\em different} from the action discussed in the previous
subsection. But, fortunately, the resulting transformation groupoids
are the same (Lemma \ref{L:coincide}). In particular,
$H^{1}(\Gamma,\bbG)$ is equal to the cokernel of $d$.

\medskip
\noindent{\em The product in $Z^{1}(\Gamma,\bbG)$.}
We define a product in $Z^{1}(\Gamma,\bbG)$ as follows. More
generally, let $C^{1}(\Gamma,\bbG)$ be the set of all $1$-cochains,
where by a
 {\bf $1$-cochain} we mean a pair $(p,\vep)$,
     $$p \:  \Gamma \to G_0, \   \ \vep \:
         \Gamma\times \Gamma \to G_1,$$
of pointed set maps. Let $(p_1,\vep_1)$ and $(p_2,\vep_2)$ be in
$C^{1}(\Gamma,\bbG)$. We define the product
$(p_1,\vep_1)\cdot(p_2,\vep_2)$ to be the pair $(p,\vep)$ where
         $$p(\si):=p_1(\si)p_2(\si),$$
      $$\vep(\si,\tau):=\vep_1(\si,\tau)^{p_2(\si\tau)}\cdot\vep_2(\si,\tau)
        \cdot\{p_2(\si),{}^{\si}\!p_1(\tau)\}^{{}^{\si}\!p_2(\tau)}.$$
It can be checked that this makes $C^{1}(\Gamma,\bbG)$ into a group.
The inverse of the element  $(p,\vep)$ in $C^{1}(\Gamma,\bbG)$ is
the pair $(q,\lambda)$ where
             $$q(\si):=p(\si)^{-1},$$
        $$\lambda(\si,\tau)=(\vep(\si,\tau)^{-1})^{p(\si\tau)^{-1}}
            \cdot\{p(\si)^{-1},{}^{\si}\!p(\tau)^{-1}\}.$$
The subset    $Z^{1}(\Gamma,\bbG) \subset C^{1}(\Gamma,\bbG)$ is indeed
a subgroup.

\medskip
\noindent{\em The group homomorphism $d \: C^{0}(\Gamma,\bbG)
\to Z^{1}(\Gamma,\bbG)$.}
Next we construct a group homomorphism 
  $$d \: C^{0}(\Gamma,\bbG) \to Z^{1}(\Gamma,\bbG).$$  
(Note that $ C^{0}(\Gamma,\bbG)$ is
endowed with   the group structure defined in
$\S$\ref{SS:braiding0}.) Let $(g,\theta)$ be in
$C^{0}(\Gamma,\bbG)$. We define $d(g,\theta)$ to be the pair
$(p,\vep)$ where
       $$p(\si):=g^{-1}\cdot{}^{\si}\!g\cdot\partial\theta(\si)^{-1},$$
       $$\vep(\si,\tau):=\theta(\si\tau)
         \cdot\big(\theta(\si)^{-1}\big)^{({}^{\si}\!g^{-1}\cdot{}^{\si\tau}\!g)}
                   \cdot{}^{\si}\!\theta(\tau)^{-1}$$
It is not difficult to  check that this is a group homomorphism.

\medskip
\noindent{\em The crossed-module $[d \: C^{0}/B^{0} 
\to Z^{1}]$.}
The group homomorphism $d$ vanishes on the subgroup
$B^{0}(\Gamma,\bbG)  \subseteq C^{0}(\Gamma,\bbG) $ of
0-coboundaries. Therefore, $d$ factors through a homomorphism
       $$d \:  C^{0}(\Gamma,\bbG)/B^{0}(\Gamma,\bbG) \to Z^{1}(\Gamma,\bbG),$$
which, by abuse of notation, we have denoted again by $d$.
There is a right action of $Z^{1}(\Gamma,\bbG)$ on
$C^{0}(\Gamma,\bbG)$ which preserves $B^{0}(\Gamma,\bbG)$ and
makes
        $$[d \:  C^{0}(\Gamma,\bbG)/B^{0}(\Gamma,\bbG) \to Z^{1}(\Gamma,\bbG)]$$
into a crossed-module. It is given by
   $$(g,\theta)^{(p,\vep)}=(g,\vart),$$
where $\vart \: \Gamma \to G_1$ is defined by
    $$\vart(\si)=\theta(\si)^{p(\si)}\cdot\{p(\si),
    {}^{\si}\!g\}\cdot\{g,p(\si)\}^{g^{-1}\cdot{}^{\si}\!g}.$$

Observe that the kernel of $d$ coincides with $H^{0}(\Gamma,\bbG)$.
We show that the cokernel of $d$ coincides with
$H^{1}(\Gamma,\bbG)$. We do so by comparing the action of
$C^{0}(\Gamma,\bbG)$ on $Z^{1}(\Gamma,\bbG)$ introduced in the
previous subsection (the one that gave rise to the groupoid
$\mcZ^{1}(\Gamma,\bbG)$ of 1-cocycles) with the multiplication
action of $C^{0}(\Gamma,\bbG)$ on $Z^{1}(\Gamma,\bbG)$ via $d$. More
precisely, we have the following.

\begin{lem}{\label{L:coincide}}
  Let $(g,\theta)$ be in  $C^{0}(\Gamma,\bbG)$ and $(p,\vep)$ in
  $Z^{1}(\Gamma,\bbG)$. Let $(p,\vep)^{(g,\theta)}$ be the action of
  $C^{0}(\Gamma,\bbG)$ on $Z^{1}(\Gamma,\bbG)$ introduced at the end of
  the previous subsection. Then,
        $$(p,\vep)^{(g,\theta)}=
          (p,\vep)\cdot d\big(g,\theta\cdot \delta(p,g)\big),$$
  where $\delta(p,g) \: \Gamma \to G_1$ is defined by
    $$\si \mapsto \{g,p(\si)\}^{g^{-1}\cdot{}^{\si}\!g}.$$
\end{lem}

\begin{cor}
    When $\bbG$ has a $\Gamma$-equivariant braiding, the first
    cohomology set $H^{1}(\Gamma,\bbG)$ inherits a natural group
    structure, $H^{0}(\Gamma,\bbG)$ is abelian, and there is a
    natural action of $H^{1}(\Gamma,\bbG)$
    on $H^{0}(\Gamma,\bbG)$.
\end{cor}

\begin{rem}
  The crossed-module $[d \: C^{0}/B^{0}\to Z^{1}]$ is a model 
  for the 2-group $\mathcal{H}^1$ defined
  by Garz\'on and del R\'io \cite{GadR}.
\end{rem}

\subsection{When braiding is symmetric}{\label{SS:symmetricbraiding1}}

In the previous subsection, we saw that when $\bbG$ has a
$\Gamma$-equivariant braiding, the first cohomology set
$H^{1}(\Gamma,\bbG)$ carries a natural group structure. This was
done by identifying $H^{1}(\Gamma,\bbG)$ with the cokernel of the
crossed-module $[d \:  C^{0}/B^{0} \to Z^{1}]$. 
In the case where the braiding is symmetric
(i.e., $\{g,h\}\{h,g\}=1$) we can do even better.

\begin{lem}{\label{L:symmetric}}
     Suppose that the braiding on $\bbG$ is symmetric (resp.,
     Picard). Then, the crossed-module $[d \:
     C^{0}(\Gamma,\bbG)/B^{0}(\Gamma,\bbG) \to Z^{1}(\Gamma,\bbG)]$
     is braided and symmetric (resp., Picard). The braiding is given
     by $$\{(p_1,\vep_1),(p_2,\vep_2)\}:=(1,\{p_2,p_1\}),$$ where
     $\{p_2,p_1\} \: \Gamma \to G_1$ is the pointwise bracket of the
     maps $p_1,p_2 \: \Gamma \to G_0$.
     (Note the reverse order.)
\end{lem}

The above braiding is obtained by unraveling the symmetry morphism
$b$ of $\S$\ref{SS:symmetricbutterfly}.

\begin{cor}
    When the braiding on $\bbG$ is symmetric, the group structure on
    $H^{1}(\Gamma,\bbG)$ is abelian.
\end{cor}

\section{The nonabelian complex $\mcK(\Gamma, \bbG)$}{\label{S:B}}

For an abelian group $G$ with an action of $\Gamma$ one can find a
chain complex $\mcK(\Gamma, G)$ whose cohomologies are
$H^i(\Gamma,G)$. The corresponding statement is obviously not true
for a nonabelian $G$ (or a crossed-module $\bbG$). However, it seems to be
true ``as much as it can be''. More precisely, even though the
complex $\mcK(\Gamma, \bbG)$ does not exist,  truncated versions of
it exist. And ``the more abelian $\bbG$ is'' the longer and the more
abelian these truncations become. 

For example, for an arbitrary
$\bbG$, there is a crossed-module in groupoids $\mcK^{\leq
1}(\Gamma, \bbG)$, concentrated in degrees $[-1,1]$, whose
cohomologies are precisely $H^i(\Gamma,\bbG)$, $i=-1,0,1$. When
$\bbG$ is braided,  $\mcK^{\leq 1}(\Gamma, \bbG)$ is actually a
2-crossed-module; in fact, it can be extended one step further to a
2-crossed-module in groupoids $\mcK^{\leq 2}(\Gamma, \bbG)$ which is
concentrated in degrees $[-1,2]$.\footnote{We will not prove this
here.}  In the case where $\bbG$ is symmetric, $\mcK^{\leq
2}(\Gamma, \bbG)$ is expected to come from a  3-crossed-module. 
We are not able to prove
this here, but we prove the weaker statement that $\mcK^{\leq
1}(\Gamma, \bbG)$ is a braided 2-crossed-module.

\subsection{The case of arbitrary $\bbG$}{\label{SS:arbitrary}}

The crossed-module in groupoids $\mcK^{\leq 1}(\Gamma, \bbG)$ that
we will define below neatly packages everything we have discussed so
far in the preceding sections.

Let $\mcZ(\Gamma,\bbG)= [Z^1(\Gamma,\bbG)\times C^0(\Gamma,\bbG)
\sst{} Z^1(\Gamma,\bbG)]$ be the groupoid of 1-cocycles defined in
$\S$\ref{S:1}.  Recall that it is the action groupoid of the right
action of $C^0$ on $Z^1$ defined in $\S$\ref{S:1}.
We define $\mcK^{\leq 1}(\Gamma, \bbG)$ to be
  $$\mcK^{\leq 1} (\Gamma, \bbG):=
       [\coprod_{c \in Z^1} G_1(c) \llra{d} \mcZ(\Gamma,\bbG)].$$
Here $Z^1=Z^1(\Gamma,\bbG)$ and $G_1(c)=G_1$.

For every 1-cocycle $c \in Z^1(\Gamma,\bbG)$, the effect of the
differential $d$ on the corresponding component $G_1(c)$ of the
disjoint union $\coprod_{c \in Z^1} G_1(c)$ is defined by
        $$\mu \mapsto (\partial_{\mu},\theta_{\mu}).$$
(See $\S$\ref{S:0} for notation.) Here, we are thinking of $
(\partial_{\mu},\theta_{\mu})$
as an arrow in $\mcZ(\Gamma,\bbG)$ going from the object  $c$ to itself.

The right action of $\mcZ(\Gamma,\bbG)$ on $\coprod_{c \in Z^1}
G_1(c)$ is defined as follows. Let $c, c' \in Z^1$ be 1-cocycles,
and $(g,\theta) \in C^0$ an arrow between
them. Then, $(g,\theta) $ acts by
     $$G_1(c) \to G_1(c'),$$
   $$\mu \mapsto \mu^{(g,\theta)}:=\mu^g.$$

The following proposition is easy to prove.
\begin{prop}{\label{P:crossedmod}}
     For $i=-1,0,1$, we have
       $$H^i(\Gamma,\bbG)=H^i\mcK^{\leq 1}(\Gamma, \bbG).$$
\end{prop}

\subsection{The case of braided $\bbG$}{\label{SS:braided}}

In the case where $\bbG$ is endowed with  a $\Gamma$-equivariant
braiding, it follows from Lemma \ref{L:coincide} that the
crossed-module in groupoids $\mcK^{\leq 1}(\Gamma, \bbG)$ comes from
(see $\S$\ref{A:1}) the  2-crossed-module 
  $$[C^{-1}(\Gamma,\bbG)
      \llra{d} C^0(\Gamma,\bbG) \llra{d} Z^1(\Gamma,\bbG)],$$ 
 where $C^{-1}(\Gamma,\bbG):=G_1$.  
(Note that the group structure on $C^0(\Gamma,\bbG)$ is the one
defined in $\S$\ref{SS:braiding0}.) The boundary maps $d$ are the
ones defined in $\S$\ref{S:0} and $\S$\ref{S:1}. That is
    $$C^{-1}(\Gamma,\bbG) \llra{d} C^0(\Gamma,\bbG),$$
       $$\mu \mapsto  (\partial_{\mu},\theta_{\mu}),$$
and
     $$C^{0}(\Gamma,\bbG) \llra{d} Z^1(\Gamma,\bbG),$$
       $$(g,\theta) \mapsto (p,\vep),$$
where
       $$p(\si):=g^{-1}\cdot{}^{\si}\!g\cdot\partial\theta(\si)^{-1},$$
       $$\vep(\si,\tau):=\theta(\si\tau)
             \cdot\big(\theta(\si)^{-1}\big)^{({}^{\si}\!g^{-1}\cdot{}^{\si\tau}\!g)}
                   \cdot{}^{\si}\!\theta(\tau)^{-1}.$$

The action of $Z^1(\Gamma,\bbG)$ on $C^{-1}(\Gamma,\bbG)$ is defined to be the trivial one.
The action of  $Z^1(\Gamma,\bbG)$ on $C^0(\Gamma,\bbG)$ is defined to be the one of
$\S$ \ref{SS:braiding1}. Namely,  
   $$(g,\theta)^{(p,\vep)}:=(g,\vart),$$
where $\vart \: \Gamma \to G_1$ is defined by
    $$\vart(\si)=\theta(\si)^{p(\si)}\cdot\{p(\si),
    {}^{\si}\!g\}\cdot\{g,p(\si)\}^{g^{-1}\cdot{}^{\si}\!g}.$$
The action of  $C^0(\Gamma,\bbG)$ on $C^{-1}(\Gamma,\bbG)$ is defined by
   $$\mu^{(g,\theta)}:=\mu^g.$$
Finally, the bracket 
  $$\{,\} \: C^0(\Gamma,\bbG) \times C^0(\Gamma,\bbG) \to C^{-1}(\Gamma,\bbG)$$
is defined by
  $$\{(g_1,\theta_1),(g_2,\theta_2)\}:=\{g_1,g_2\}.$$

By abuse of notation, we denote the above 2-crossed-module again by
$\mcK^{\leq 1}(\Gamma, \bbG)$. By  Proposition \ref{P:crossedmod},
the cohomologies of this 2-crossed-module are naturally isomorphic
to $H^i(\Gamma,\bbG)$,
$i=-1,0,1$.

\subsection{The case of symmetric $\bbG$}{\label{SS:symmetric}}

In the case where the braiding on $\bbG$ is symmetric, the
2-crossed-module $\mcK^{\leq 1}(\Gamma, \bbG)$ is braided in the sense of
$\S$\ref{A:1}. We prove this using the following lemma.

\begin{lem}{\label{L:barided2cm}}
 Let $\mathbb{C}=[K \llra{\partial} L \llra{\partial} M]$ be a 
 2-crossed-module such that the action
 of $M$ on $K$ is trivial and the bracket $\{,\}\: L\times L \to K$ is symmetric 
 (i.e., $\{g,h\}\{h,g\}=1$, for every $g,h \in L$). Assume that we are given a 
 bracket $\{,\} \: M\times M \to L$ which satisfies the following conditions:
   \begin{itemize}
       \item for every $x,y \in M$, $\partial\{x,y\}=x^{-1}y^{-1}xy$,
       \item for every $g \in L$ and $x \in M$, $\{\partial g,x\}=g^{-1}g^x$
               and  $\{x,\partial g\}=(g^{-1})^xg$.
   \end{itemize}
   With the notation of $\S$\ref{A:1},
   let the brackets $\{,\}_{(1,0)(2)}$,  $\{,\}_{(2,0)(1)}$, $\{,\}_{(0)(2,1)}$, $\{,\}_{(0)(2)}$
   be the trivial ones (i.e., their value is always $1$). Let $\{,\}_{(2)(1)} \: L\times L \to K$ 
   be the given bracket of $\mathbb{C}$, and define 
   $\{,\}_{(1)(0)} \: L\times L \to K$   by   $\{g,h\}_{(1)(0)}:=\{h,g\}_{(2)(1)}$. 
   Then, $\mathbb{C}$
   is a braided 2-crossed-module in the sense of $\S$\ref{A:1}.
\end{lem}

\begin{proof}
  All axioms ($\mathbf{3CM1}$)-($\mathbf{3CM18}$)  
  of (\cite{ArKuUs}, Definition 8) follow trivially from our assumptions, except for 
  $\mathbf{3CM6}$. For this, we must prove that 
        $$\{g,h\}^{[g,h]}=\{g,h\}$$ 
  for every $g,h \in L$.
  Here, $[g,h]:=g^{-1}h^{-1}gh$. We have, $[g,h]=\partial\{g,h\}(h^{-1})^{\partial g}h$.
  Note that the assertion is true if we replace $[g,h]$ by $\partial\{g,h\}$. So we have to show
  that $\{g,h\}^{(h^{-1})^{\partial g}h}=\{g,h\}$. This is true 
  because the action of $M$ on $K$ is trivial and $\partial \: K \to L$ is $M$-equivariant.
\end{proof}

\begin{rem}
  Note that we are using modified versions of axioms of (\cite{ArKuUs}, 
  Definition 8) because  our conventions for the actions (left or right) and the brackets, hence also 
  our 2-crossed-module axioms, are different from those of {\em loc.~cit}.
\end{rem}  

Now,  in the above lemma take $\mathbb{C}$ to be $\mcK^{\leq 1}(\Gamma, \bbG)$.
Let
  $$\{,\} \: Z^1(\Gamma,\bbG) \times Z^1(\Gamma,\bbG) \to C^0(\Gamma,\bbG).$$
be the braiding  defined  in Lemma \ref{L:symmetric}. Namely,
    $$\{(p_1,\vep_1),(p_2,\vep_2)\}:=(1,\{p_2,p_1\}).$$
Here, $\{p_1,p_2\} \: \Gamma \to G_1$
is the pointwise bracket of the maps $p_1,p_2 \: \Gamma \to G_0$. It is not difficult to check
that this bracket satisfies the two conditions of the above lemma. This endows 
$\mcK^{\leq 1}(\Gamma, \bbG)$ with the structure of a braided 2-crossed-module.

\subsection{Invariance under equivalence}{\label{SS:invariance}}

The crossed-module in groupoids $\mcK^{\leq 1}(\Gamma, \bbG)$ is
clearly functorial in (strict) $\Gamma$-equivariant morphisms $f \:
\bbH \to \bbG$ of crossed-modules. (Also, in the braided case, the
2-crossed-module $\mcK^{\leq 1}(\Gamma, \bbG)$ is functorial in
strict $\Gamma$-equivariant braided morphisms of crossed-modules.)
In this subsection, we prove that if $f$ is an equivalence of
crossed-modules (i.e., induces isomorphisms
on cohomologies), then the induced morphism
    $$f_* \:  \mcK^{\leq 1}(\Gamma, \bbH) \to \mcK^{\leq 1}(\Gamma, \bbG)$$
is also an equivalence  (i.e., induces isomorphisms on
cohomologies).

We begin by a definition. Let
  $$\bbG=[\partial \: G_1 \to G_0]$$
be a crossed-module, and $p \: G_0' \to G_0$  a group homomorphism
from a certain group $G_0'$. Let $G_1' := G_1\times_{G_0}G_0'$ be
the fiber product. There is a natural crossed-module structure on
     $$\bbG':=[\pr_2 \: G'_1 \to G'_0].$$
We call this the
{\em pullback} crossed-module via $p$ and denote it by $p^*\bbG$.
We have a natural projection $P \: \bbG' \to \bbG$.

The next lemmas will be used in the proof of Proposition
\ref{P:equivalence}.

\begin{lem}{\label{L:pullback}}
   Notation being as above, assume that  the images of $p$ and
   $\partial$ generate $G_0$. Then  $P \: \bbG' \to \bbG$ is an
   equivalence of crossed-modules. Conversely, if $P \: \bbG' \to
   \bbG$ is an equivalence of crossed-modules, then the  images of
   $P_0 \: G'_0 \to G_0$ and $\partial$ generate $G_0$, and $\bbG'$
   is naturally isomorphic to $P_0^*\bbG$.
\end{lem}

\begin{proof}
  Easy.
\end{proof}

\begin{lem}{\label{L:surj}}
     Let $F \: \bbH \to \bbG$ be an equivalence of
     $\Gamma$-crossed-modules. Then, there is a commutative diagram
     $$\xymatrix{ \bbH \ar@/^1pc/ [rr]^F  & \bbH' \ar[r]_{F'}
     \ar[l]^P& \bbG}$$ of equivalences of $\Gamma$-crossed-modules
     such that $P_0$ and $F'_0$ are surjective. In particular, by
     Lemma \ref{L:pullback}, $\bbH'$ is naturally isomorphic to both
     $P_0^*\bbH$ and $(F'_0)^{*}\bbG$.
\end{lem}

\begin{proof}
  Consider the right action of $H_0$ on $G_1$ via $F_0 \: H_0 \to
  G_0$, and form the semidirect product $H_0\ltimes G_1$. It acts on
  $H_1\times G_1$ on the right by the rule
  $$(\be,\al)^{(h,\gamma)}:=(\be^h,\gamma^{-1}\al^{P_0(h)}\gamma).$$
  With this action, we obtain a $\Gamma$-crossed-module
  $$\bbH':=[\partial \: H_1\times G_1 \to H_0\ltimes G_1],$$
  $$\partial(\be,\al):=(\partial_{\bbH}\be, F_1(\be^{-1})\al).$$ We
  define  $P \: \bbH' \to \bbH$ to be the first projection map
  $(\pr_1,\pr_1)$, and  $F' \: \bbH' \to \bbG$ to be $(\pr_2,\rho)$,
  where $\rho \: H_0\ltimes G_1 \to G_0$ is defined by $(h,\al)
  \mapsto P_0(h)\partial_{\bbG}\al$.
  It is easy to verify that $P$ and $F'$ satisfy the desired properties.
\end{proof}

We now come to the proof of invariance of  $\mcK^{\leq 1}(\Gamma,
\bbG)$ under equivalences.

\begin{prop}{\label{P:equivalence}}
      Let $f \: \bbH \to \bbG$ be a $\Gamma$-equivariant morphism of
      crossed-modules which is  an equivalence (i.e., induces
      isomorphisms on $\ker\partial$ and $\coker\partial$). Then,
      $$f_* \:  \mcK^{\leq 1}(\Gamma, \bbH) \to \mcK^{\leq
      1}(\Gamma, \bbG)$$ is  an equivalence of crossed-modules in
      groupoids  (i.e., induces isomorphisms on cohomologies). In
      particular, the induced maps $H^i(\Gamma,\bbH) \to H^i(\Gamma,\bbG)$
      are isomorphisms for $i=-1,0,1$.
\end{prop}

\begin{proof}
   By Lemma \ref{L:surj}, we may assume that $f_0 \: H_0 \to G_0$ is
   surjective. Therefore, by Lemma \ref{L:pullback}, we may assume
   that $\bbH$ is the pullback of $\bbG$ along $f_0$. That is,
   $\bbH=[\pr_1 \: H_0\times_{G_0}G_1\to H_0]$ and $f$ is $(\pr_2,
   f_0) \: [H_0\times_{G_0}G_1\to H_0] \to [G_1 \to G_0]$. We
   calculate $\mcK^{\leq 1}(\Gamma, \bbH)$ explicitly and show that
   $f_*$ induces isomorphisms on cohomologies.

   By definition ($\S$\ref{SS:arbitrary}), we have
     $$\mcK^{\leq 1} (\Gamma, \bbH)=
          \coprod_{c \in Z^1} H_0\times_{G_0}G_1(c) \llra{d} \mcZ(\Gamma,\bbH),$$
   where
     $$\mcZ(\Gamma,\bbH)=
    [Z^1(\Gamma,\bbH)\times C^0(\Gamma,\bbH) \sst{} Z^1(\Gamma,\bbH)].$$
    We calculate $Z^1(\Gamma,\bbH)$ and $C^0(\Gamma,\bbH)$ as follows.
    An element in $Z^1(\Gamma,\bbH)$
    is a pair $(p,\vep)$ where
      $$p\: \Gamma \to H_0 \ \ \text{and} \ \ \vep \:
                                  \Gamma \times \Gamma \to G_1$$
    are pointed set maps satisfying the following conditions:
 \begin{itemize}
     \item for every
        $\si, \tau \in \Gamma$,
            $f_0p(\si\tau)\cdot\partial\vep(\si,\tau)=
           f_0p(\si)\cdot{}^{\si}\!f_0p(\tau)$,
     \item for every $\si, \tau, \ups \in \Gamma$,
          $\vep(\si,\tau\ups)\cdot{}^{\si}\!\vep(\tau,\ups)=
           \vep(\si\tau,\ups)\cdot\vep(\si,\tau)^{{}^{\si\tau}\!p(\ups)}$.
 \end{itemize}
An element in    $C^0(\Gamma,\bbH)$ is a triple
$\big(h,(\theta_1,\theta_2)\big)$, with $h \in H_0$, $\theta_1 \:
\Gamma \to H_0$, and $\theta_2 \: \Gamma \to G_1$ pointed set maps
such that  $f_0\theta_1=\partial\theta_2$.

The map of groupoids $\mcZ(\Gamma, \bbH)\to \mcZ(\Gamma, \bbG)$ is
the one induced by the following maps:
      $$Z^1(\Gamma,\bbH) \to Z^1(\Gamma,\bbG),$$
        $$(p,\vep) \mapsto (f_0p,\vep),$$
       $$C^0(\Gamma,\bbH) \to C^0(\Gamma,\bbG),$$
        $$\big(h,(\theta_1,\theta_2)\big)
                            \mapsto \big(f_0(h),\theta_2\big).$$
Since $f_0$ is surjective, we see immediately that $\mcZ(\Gamma,
\bbH)\to \mcZ(\Gamma, \bbG)$ is surjective on objects and that it is
a fibration of groupoids (i.e., has the arrow lifting property).
This  almost proves that the induced map on the set of isomorphism
classes of these groupoids is a bijection (which is the same thing
as saying that $f_*$ induces a bijection on the first cohomology
sets). All we need to check is that if $(p,\vep), (p',\vep') \in
Z^1(\Gamma,\bbH)$ map to the same element in $Z^1(\Gamma, \bbG)$,
then they are joined by an arrow in $\mcZ(\Gamma, \bbH)$. We have
$\vep=\vep'$ and $f_0p=f_0p'$. Hence, the map $\theta \: \Gamma \to
H_0$ defined by $\si \mapsto p(\si)^{-1}p'(\si)$ factors through
$\ker f_0$. It is easy to see that $\big(1,(1,\theta)\big) \in
C^0(\Gamma,\bbH)$ provides the desired arrow in $\mcZ(\Gamma, \bbH)$
joining $(p,\vep)$  to $(p',\vep')$. This completes the proof
 that $f_*$ is a bijection on the first cohomology sets.

To show that $f_*$ induces an isomorphism on $H^{-1}$ and $H^0$, we
need to verify that the induced map of crossed-modules (see
$\S$\ref{S:0})
  $$[H_0\times_{G_0}G_1 \to Z^0(\Gamma,\bbH)]\ \to \
              [G_1 \to Z^0(\Gamma,\bbG)]$$
is an equivalence.  Observe that $Z^0(\Gamma,\bbH) \subset
C^0(\Gamma,\bbH)$ consists of triples
$\big(h,(\theta_h,\theta_2)\big)$, where $h$ and $\theta_2$ are
arbitrary and $\theta_h \:\Gamma \to H_0$ is defined by the rule
$\si \mapsto h^{-1}\cdot{}^{\si}\!h$. The map $Z^0(\Gamma,\bbH) \to
Z^0(\Gamma,\bbG)$ is given by $\big(h,(\theta_h,\theta_2)\big)
\mapsto \big(p_0(h),\theta_2\big)$. It is clear that this map is
surjective.

By Lemma \ref{L:pullback}, it is enough to show that the following
diagram is cartesian
        $$\xymatrix{H_0\times_{G_0}G_1 \ar[r]^{d'} \ar [d]_{\pr_2}
              &  Z^0(\Gamma,\bbH) \ar[d] \\
           G_1 \ar[r]_{d} & Z^0(\Gamma,\bbG)}$$
The fact that this diagram is cartesian becomes obvious once we
recall ($\S$\ref{S:0}) that $d$ and $d'$ are defined as follows:
  $$d(\mu)=(\partial\mu, \theta_\mu),
                \ d'(h,\al)=\big(h,(\theta_h,\theta_{\al})\big).$$
The proof of the proposition is  complete.
\end{proof}

\begin{rem}
    The butterfly approach of $\S$\ref{S:Butterfly} provides another
    proof of Proposition \ref{P:equivalence}.
\end{rem}

\section{Butterflies}{\label{S:Butterfly}}

This section is a prelude to $\S$\ref{S:CohviaButterfly} in which we
will present an alternative construction of the cohomologies
$H^i(\Gamma,\bbG)$ and also of  the complex $\mcK^{\leq
1}(\Gamma,\bbG)$. This new construction, which is based on the idea
of {\em butterfly}, has the following advantages: 1) it is much
easier to write down $H^i(\Gamma,\bbG)$ and describe their
properties, 2) it is easy to recover the cocycles from this
description, 3) it works for  arbitrary topological groups, and in
fact in any topos, 4) it can be generalized to
the case where $\Gamma$ itself is a crossed-module.

To motivate the relevance of butterflies, let us explain the idea in
the case of $H^1$. Assume for the moment that $\bbG=G$ is a group.
In this case, to give a  1-cocycle (e.g, a crossed-homomorphism) $p
\: \Gamma \to G$ is  the same thing as giving a group homomorphism
$\tilde{p} \: \Gamma \to G\rtimes\Gamma$ making the following
diagram commutative:
   $$\xymatrix{ & G\rtimes\Gamma \ar[d]^{\pr} \\
          \Gamma \ar[r]_{\id}  \ar[ru]^{\tilde{p}}  & \Gamma}$$
The (right) conjugation action of $G  \subseteq G\rtimes\Gamma$  on
$G\rtimes\Gamma$ induces an action on the set of such $\tilde{p}$.
The transformation groupoid of this action is what we called
$\mcZ^1(\Gamma,G)$ in $\S$\ref{S:1}. The set of isomorphism classes
of $\mcZ^1(\Gamma,G)$ is  $H^1(\Gamma,G)$.

The aim is now to imitate this definition in the case where $G$
replaced by a crossed-module $\bbG$. A group homomorphisms
$\tilde{p} \:\Gamma \to G\rtimes\Gamma$ should now be replaced by a
{\em weak} morphism of crossed-modules
$\Gamma \to \bbG\rtimes\Gamma$. This is where butterflies come
in the picture.

\subsection{Butterflies}{\label{SS:butterfly}}

We recall the definition of a butterfly from \cite{Maps}.

Let $\bbG=[G_1\to G_0]$ and $\bbH=[H_1\to H_0]$ be crossed-modules.
By a {\em butterfly} from $\bbH$ to $\bbG$ we mean a commutative
diagram of groups
   $$\xymatrix@C=8pt@R=6pt@M=6pt{ H_1 \ar[rd]^{\kappa} \ar[dd]
                          & & G_1 \ar[ld]_{\iota} \ar[dd] \\
                            & E \ar[ld]^{\pi} \ar[rd]_{\rho}  & \\
                                       H_0 & & G_0       }$$
such that the two diagonal maps are complexes and the NE-SW diagonal
is short exact. We require that for every $x \in E$, $\al \in G$ and
$\be \in H$,
    $$\iota(\al^{\rho(x)})=x^{-1}\iota(\al)x \ \text{and} \
    \kappa(\be^{\pi(x)})=x^{-1}\kappa(\be)x.$$
We denote the above butterfly by the 5-tuple
$(E,\rho,\pi,\iota,\kappa)$, or if there is no fear of confusion,
simply by $E$.

A {\em morphism} between two butterflies $(E,\rho,\pi,\iota,\kappa)$
and $(E',\rho',\pi',\iota',\kappa')$ is a pair $(t,g)$ where $g \in
G_0$ and $t \: E \to E'$ is an isomorphism of groups. We require
that $t$ commutes with the $\kappa$ and the $\pi$ maps and satisfies
the relations
        $$g^{-1}\rho(x)g=\rho'(t(x))
             \ \text{and} \ \iota'(\al^g)=t\iota(\al)$$
for every  $x \in E$, $\al \in G_1$. The composition of two arrows
$(t,g) \: E \to E'$ and $(t',g') \: E' \to E''$ is defined to be
$(t'\circ t,gg')$.

A {\em 2-morphism} between $(g,t)$ and $(g',t')$ is an element $\mu
\in G_1$ such that
    $$g\partial(\mu)=g' \ \text{and} \ t'=\mu^{-1}t\mu.$$
The composition of two 2-arrows $\mu_1$ and $\mu_2$ is defined to be
$\mu_1\mu_2$.

For fixed $\bbH$ and $\bbG$, the butterflies between them are
objects of a 2-groupoid whose morphisms and 2-arrows are defined as
above.

\begin{ex}{\label{E:Dedecker}}
    Assume that $\bbH=\Gamma$ is a group. Then, a butterfly from
    $\Gamma$ to $\bbG$  is a diagram
     $$\xymatrix@C=8pt@R=6pt@M=6pt{
                          & & G_1 \ar[ld]_{\iota} \ar[dd]^{\partial} \\
                            & E \ar[ld]_{\pi} \ar[rd]^{\rho}  & \\
                                      \Gamma & & G_0       }$$
    where the diagonal sequence is short exact and the map $\rho$
    intertwines the conjugation action of $E$ on $G_1$ with the
    crossed-module action of $G_0$. Such a diagram corresponds to a
    weak morphism $\Gamma \to \bbG$ and also to a 1-cocycle on
    $\Gamma$ with values in $\bbG$ (for the trivial action of
    $\Gamma$). A morphism between two such diagrams corresponds to
    a transformation of weak functors and
     also to an equivalence of 1-cocycles.
\end{ex}

\begin{ex}{\label{E:braided}}
     A braided crossed-module $\bbG$ is the same things as a group
     object in the category of crossed-modules and weak morphisms.
     More precisely, the multiplication morphism
     of this group object is given by the butterfly
         $$\xymatrix@C=8pt@R=6pt@M=6pt{ G_1\times G_1 \ar[rd]^{k}
           \ar[dd]_{(\partial,\partial)}
                          & & G_1 \ar[ld]_{i} \ar[dd]^{\partial} \\
                            & B \ar[ld]^{p} \ar[rd]_{r}  & \\
                                       G_0\times G_0 & & G_0       }$$
     where the group $B$ is defined as follows. The underlying set
     of $B$ is $G_0\times G_0\times G_1$. The product in $B$ is defined by
         $$(g,h,\al)\cdot(g',h',\al'):=(gg',hh',\{h,g'\}^{h'}\al^{g'h'}\al').$$
     The structure maps of the butterfly are given by
          $$k(\al,\be):=(\partial \al,\partial\be,\be^{-1}\al^{-1}),\ \  \
               i(\al):=(1,1,\al)$$
          $$p(g,h,\al):=(g,h), \ \  r(g,h,\al):=gh\partial\al,$$
\end{ex}

\section{Cohomology via butterflies}{\label{S:CohviaButterfly}}

In this section we will use the idea discussed at the beginning of
the previous section to
give a simple description of the cohomologies $H^i(\Gamma,\bbG)$.

Given a group $\Gamma$ acting on a crossed-module $\bbG$ on the
left, the semi-direct product $\bbG\rtimes\Gamma$ is the
crossed-module
    $$\bbG\rtimes\Gamma:=[(\partial,1) \: G_1 \to G_0\rtimes\Gamma].$$
The action of $G_0\rtimes\Gamma$ on $G_1$ is defined by
   $$\al^{(g,\si)}:={}^{\si^{-1}}\!(\al^g)=
      ({}^{\si^{-1}}\!\al)^{({}^{\si^{-1}}\!g)}.$$
This crossed-module comes with a natural projection map
   $$\pr \: \bbG\rtimes\Gamma \to \Gamma.$$
Here, by abuse of notation, we have denoted the crossed-module $[1\to \Gamma]$ by
$\Gamma$.

\subsection{Butterfly description of $H^i(\Gamma,\bbG)$}{\label{SS:H1new}}

We define the 2-groupoid $\mfZ(\Gamma,\bbG)$ as follows. The set of
objects of $\mfZ(\Gamma,\bbG)$ are pairs $(E,\rho)$ where $E$ is an
extension
   $$1 \to G_1 \llra{\iota} E \llra{\pi} \Gamma \to 1$$
and $\rho \: E  \to     G_0$ is a map which makes the diagram
     $$\xymatrix@C=8pt@R=6pt@M=6pt{
                          & & G_1 \ar[ld]_{\iota} \ar[dd]^{\partial} \\
                            & E  \ar[ld]_{\pi} \ar[rd]^(0.4){\rho}  & \\
                                      \Gamma & &  G_0       }$$
commute and satisfies the following conditions:
\begin{itemize}
   \item for every $x,y \in E$,  $\rho(xy)=\rho(x)\cdot {}^{\pi(x)}\!\rho(y)$,
   \item for every $x \in E$ and $\al \in G_1$,
       $\iota\big({}^{\pi(x)^{-1}}\!(\al^{\rho(x)})\big)=x^{-1}\iota(\al)x$.
\end{itemize}
(The maps $\iota$ and $\pi$ are also part of the data but we
suppress them from the notation. We usually
identify $G_1$ with $\iota(G_1) \subseteq E$ and denote
$\iota(\mu)$ simply by $\mu$.)

An arrow in $\mfZ(\Gamma,\bbG)$ from $(E,\rho)$ to $(E',\rho')$ is a
pair $(t,g)$ where $g \in G_0$ and  $t$ is an isomorphism
$t \: E \to E'$   such that
   \begin{itemize}
       \item $\pi=\pi'\circ t$,
       \item   For every $x \in E$,
                 $g^{-1}\cdot\rho(x)\cdot{}^{\pi(x)}\!g=\rho' t(x)$,
       \item   For every $\al \in G_1$,  $\iota'(\al^g)=t\iota(\al)$.
   \end{itemize}
The composition of two arrows $(t,g) \: (E,\rho) \to (E',\rho')$ and
$(t',g') \: (E',\rho') \to (E'',\rho'')$ is defined to be $(t'\circ t,gg')$.

A 2-arrow $(t,g) \Rightarrow (t',g')$ is an element $\mu \in G_1$
such that $g\partial(\mu)=g'$ and $t'=\mu^{-1}t\mu$. The composition
of the two 2-arrows $\mu \: (t,g) \Rightarrow (t',g')$ and $\mu' \:
(t',g') \Rightarrow (t'',g'')$ is defined to be $\mu\mu'$.

The 2-groupoid is naturally pointed. The base object is
$(E_{triv},\rho_{triv})$, where $E_{triv}=G_1\rtimes\Gamma$ and
$\rho_{triv} \: E_{triv} \to G_0$
sends $(\al,\si)$ to $\partial(\al)$.

Let us now explain how $H^i(\Gamma,\bbG)$,
$i=-1,0,1$, can be recovered from the 2-groupoid $\mfZ(\Gamma,\bbG)$.

The group $H^{-1}(\Gamma,\bbG)$ is naturally isomorphic to the group
of  2-arrows from the arrow $(\id_{E_{triv}},1_{G_0})$ to itself.

The group $H^{0}(\Gamma,\bbG)$ is naturally isomorphic to the group
of 2-isomorphism classes of arrows from  the base object
$(E_{triv},\rho_{triv})$ to itself.

The pointed set $H^{1}(\Gamma,\bbG)$ is naturally isomorphic to the
pointed set of isomorphism classes of objects in
$\mfZ(\Gamma,\bbG)$.

We can also describe the groupoid $\mcZ^{1}(\Gamma,\bbG)$ defined in
$\S$\ref{S:1}. This groupoid is naturally equivalent to the groupoid
obtained by identifying 2-isomorphic arrows in $\mfZ(\Gamma,\bbG)$.

\begin{rem}{\label{R:reformulation}}
  By associating the one-winged butterfly
        $$\xymatrix@C=8pt@R=6pt@M=6pt{
                          & & G_1 \ar[ld]_{\iota} \ar[dd]^{(\partial,1)} \\
                            & E  \ar[ld]_{\pi} \ar[rd]^(0.4){(\rho,\pi)}  & \\
                                      \Gamma & &  G_0\rtimes \Gamma       }$$
  to an object in $\mfZ(\Gamma,\bbG)$, the 2-groupoid
  $\mfZ(\Gamma,\bbG)$ defined above is seen to be isomorphic to the
  2-groupoid of butterflies from $\Gamma$ to $\bbG\rtimes \Gamma$
  whose composition with the projection map $\bbG\rtimes \Gamma \to
  \Gamma$ is equal to the identity map $\Gamma \to \Gamma$. Note
  that, in contrast with \cite{Maps}, here we are considering {\em
  non pointed} transformation between butterflies. That is why we
  obtain a 2-groupoid (rather than a
  groupoid) of butterflies.
\end{rem}

\subsection{Relation to the cocycle description of $H^i$}{\label{SS:relation}}

To see how to recover a cocycle in the sense of $\S$\ref{S:1} from
the pair $(E,\rho)$, {\em choose} a set theoretic section $s \:
\Gamma \to E$ to the map $\pi$. Assume $s(1)=1$. Define $p \: \Gamma
\to G_0$ to be the composition $\rho \circ s$ and $\vep \: \Gamma
\times \Gamma \to G_1$ to be $$\vep \: (\si,\tau) \ \mapsto  \
s(\si\tau)^{-1}s(\si)s(\tau).$$ The pair $(p,\vep)$ is a 1-cocycle
in the sense of $\S$\ref{S:1}. Conversely, given a 1-cocycle
$(p,\vep)$ in the sense of $\S$\ref{S:1}, we define $E$ to be  the
group that has $\Gamma \times G_1$ as
  the underlying  set and whose product is defined by
     $$(\si_1,\al_1)\cdot(\si_2,\al_2)
     :=\big(\si_1\si_2\, ,\,
       \vep(\si_1,\si_2)\cdot{}^{\si_2^{-1}}\!(\al_1^{p(\si_2)})\al_2\big).$$
Define the group homomorphism $\rho \: E \to G_0$   by
      $$\rho(\si,a)=p(\si)\partial({}^{\si}\!g).$$
The homomorphisms $\iota \: G_1 \to E$ and $\pi \: E \to \Gamma$
are the inclusion and the projection maps on the corresponding  components.

\subsection{Group structure on $\mfZ(\Gamma,\bbG)$; the preliminary version}
{\label{SS:H1group}}

In the case where $\bbG$ has a $\Gamma$-equivariant braiding, there
is a group structure on the 2-groupoid $\mfZ(\Gamma,\bbG)$ which
lifts the one on $H^1(\Gamma,\bbG)$ introduced in
$\S$\ref{SS:braiding1}. In this subsection, we illustrate this
product  by making use of the butterfly of Example
\ref{E:braided}. In $\S$\ref{SS:H1explicit}, we give an explicit
formula for  it.

We begin by observing that the butterfly of Example
\ref{E:braided} can be augmented to a butterfly
       $$\xymatrix@C=8pt@R=6pt@M=6pt{ G_1\times G_1
          \ar[rd]^{(k,1)} \ar[dd]_{(\partial,\partial,1)}
              & & G_1 \ar[ld]_{(i,1)} \ar[dd]^{\partial} \\
            & B\rtimes\Gamma \ar[ld]_{(p,\id)} \ar[rd]^{(r,\id)}  & \\
              (G_0\times G_0)\rtimes\Gamma & & G_0\rtimes\Gamma       }$$
This butterfly gives rise to a product on   $\mfZ(\Gamma,\bbG)$ as
follows. Given  $(E,\rho)$ and $(E',\rho')$, we can think of them as
one-winged butterflies from $\Gamma$ to $\bbG\rtimes\Gamma$ relative
to $\Gamma$ (see Remark \ref{R:reformulation}). Form the one-winged
diagonal butterfly from $\Gamma$ to the fiber product of
$\bbG\rtimes\Gamma$ with itself relative to $\Gamma$. That is,
consider
         $$\xymatrix@C=8pt@R=6pt@M=6pt{ && G_1\times G_1
              \ar[ld] _{(\iota,\iota')}\ar[dd]^{(\partial,\partial,1)}  \\
                     &  E\times_{\Gamma} E' \ar[rd]^{} \ar[ld]_{} & \\
                        \Gamma    & &      (G_0\times G_0)\rtimes\Gamma    }$$
where $E\times_{\Gamma} E'$ stands for the fiber product of $E$ and
$E'$ over $\Gamma$. Composing this butterfly with the one of the
beginning of this subsection, we find a one-winged butterfly
          $$\xymatrix@C=10pt@R=8pt@M=6pt{ && G_1
              \ar[ld]_j \ar[dd]^{(\partial,1)}       \\
                &  E\star E' \ar[rd]^{(\rho\star\rho',\pi)} \ar[ld]_{} & \\
                       \Gamma    & &       G_0\rtimes\Gamma    }$$
This  is the sought after product $(E,\rho)\star(E',\rho')$. (In
$\S$\ref{SS:H1explicit} we will explicitly write down what
$(E,\rho)\star(E',\rho')$ is.)

The above product makes $\mfZ(\Gamma,\bbG)$ into a (weak) group
object in the category of 2-groupoids and weak functors. Therefore,
$\mfZ(\Gamma,\bbG)$ corresponds to a (weak) 3-group. Homotopy
theoretically, a 3-group is equivalent to a 2-crossed-module. The
2-crossed-module $\mcK^{\leq 1}(\Gamma,\bbG)$ that we encountered in
$\S$\ref{SS:braided} is a model for the 3-group $\mfZ(\Gamma,\bbG)$.
More precisely, the construction introduced at the end of
$\S$\ref{SS:H1new} gives an  equivalence  from the 3-group
associated to  $\mcK^{\leq 1}(\Gamma,\bbG)$ to $\mfZ(\Gamma,\bbG)$.

If we identify 2-isomorphic arrows in $\mfZ(\Gamma,\bbG)$, we obtain
a (weak) group object in the category of groupoids, i.e., a (weak)
2-group. This 2-group is, in turn, equivalent to  the 2-group
associated to the crossed-module
    $$[d \:  C^{0}(\Gamma,\bbG)/B^{0}(\Gamma,\bbG) \to Z^{1}(\Gamma,\bbG)],$$
namely   to $\mcZ(\Gamma,\bbG)$  ($\S$\ref{S:1}). The set
$H^1(\Gamma,\bbG)$ of  isomorphism classes of $\mfZ(\Gamma,\bbG)$
also inherits a group structure. This group structure
coincides with the one defined in  $\S$\ref{SS:braiding1}.

\subsection{Group structure on $\mfZ(\Gamma,\bbG)$; the explicit version}
{\label{SS:H1explicit}}

In this subsection, we explicitly write down the multiplication in
$\mfZ(\Gamma,\bbG)$. The formulas are obtained by unraveling the
definition of the composition of butterflies (\cite{Maps},
$\S$10.1). It is more or less straightforward how to derive the
formulas, but if done naively one usually ends up with very involved
expressions. Some extra algebraic manipulation is needed to bring
the formulas to the form presented below.

Let us start with the product of two objects in $\mfZ(\Gamma,\bbG)$.
The product $(E,\rho)\star (E',\rho')$ is the pair $(E\star
E',\rho\star\rho')$ which is defined as follows. Let
$F:=E\times_{\Gamma} E'$ be the fiber product of $E$ and $E'$ over
$\Gamma$. We endow $F$ with the following group structure:
  $$(x_1,y_1)\cdot(x_2,y_2):=
      \big(x_1x_2\cdot\iota\{\urho'(y_1)^{-1},
        {}^{\pi(x_1)}\!\urho(x_2)\}^{-1},y_1y_2\big).$$

There is a normal subgroup of $F$ consisting of elements of the form
$(\iota(\al),\iota'(\al)^{-1})$, $\al \in G_1$. We define $E\star
E'$ to be the quotient of $F$ by this normal subgroup.
Alternatively, one can think of $E\star E'$ as the group obtained
from $F$ by declaring $(\iota(\al),1)$ equal to $(1,\iota'(\al))$,
for every $\al \in G_1$. There is a natural group homomorphism $j \:
G_1 \to  E\star E'$ which sends $\al$ to the common value of
$(\iota(\al),1)$ and $(1,\iota'(\al))$. The group homomorphism
 $\rho\star\rho' \:  E\star E' \to G_0$ is defined by
        $$\rho\star\rho' \: (x,y) \mapsto \urho(x)\urho'(y).$$
This completes the definition of the product $(E,\rho)\star
(E',\rho')$.

Calculating the product of two arrows turns out to be more
complicated, and the formula is rather unpleasant, as we will now
see. Given two arrows $(t,g) \: (E_1,\rho_1) \to (E_2,\rho_2)$ and
$(t',g') \: (E'_1,\rho'_1) \to (E'_2,\rho'_2)$ in
$\mfZ(\Gamma,\bbG)$, we define their product to be the arrow
   $$(t\star t',gg') \: (E_1,\rho_1)\star(E'_1,\rho'_1)
        \to (E_2,\rho_2)\star(E'_2,\rho'_2),$$
where $t\star t'$ is the homomorphism
    $$t\star t' \: E_1\star E'_1 \to E_2\star E'_2,$$
        $$ (x,y) \mapsto \big(\{g',\urho_1(x)^{-1}g\}
              \{\urho'_1(y)g',{}^{\pi(x)}\!g^{-1}\}^{\urho_1(x)^{-1}g}
                  \cdot t(x)\ ,\ t'(y)\big).$$
The formula takes the much simpler form of
  $$(x,y) \mapsto \big(t(x),t'(y)\big)$$
in the case where  $g=g'=1$. But in general it seems our formula
can not be simplified further.

Finally, if we have two 2-arrows $(t_1,g_1) \Rightarrow (t_2,g_2)$
and $(t'_1,g'_1) \Rightarrow (t'_2,g'_2)$ given by $\mu,\mu' \in
G_1$, their product is defined by $\mu^{g'_1}\mu'$.

\subsection{The case of a symmetric braiding}{\label{SS:symmetricbutterfly}}
In the case where the braiding on $\bbG$ is symmetric, $\mfZ(\Gamma,\bbG)$
inherits a symmetric braiding
  $$b_{(E,\rho), (E',\rho')} \: (E,\rho)\star (E',\rho')
      \to (E',\rho')\star (E,\rho)$$
which is defined by
  $$(x,y) \mapsto \big(\iota\{\urho(x)^{-1},\urho'(y)^{-1}\}y\ ,\ x\big).$$
This braiding  is symmetric in the sense that
   $$b_{(E,\rho),(E',\rho')}\circ b_{(E',\rho'),(E,\rho')}=
                     \id_{(E,\rho),(E',\rho')}.$$
Since  $\mfZ(\Gamma,\bbG)$ is a group object in {\em 2-groupoids},
there is one more piece of data that goes into the definition of a
braiding on it. Given two arrows $(t,g) \: (E_1,\rho_1) \to
(E_2,\rho_2)$ and $(t',g') \: (E'_1,\rho'_1) \to (E'_2,\rho'_2)$ in
$\mfZ(\Gamma,\bbG)$, we need a 2-arrow $\psi$ making the following
diagram commute. (To make the diagram less involved, we abbreviate
$(E,\rho)$ to $E$.)
       $$\xymatrix@=-3pt@M=8pt{ E_1\star E_1' \ar[ddd]_{(t,g)\star(t'.g')}
          \ar[rrr]^{b_{E_1,E'_1}} &&&
               E_1'\star E_1 \ar[ddd]^{(t',g')\star(t,g)} \\
                        &&& \ar@{=>} [lldd]_{\psi}  \\
                            &&&     \\
                     E_2\star E_2'  \ar[rrr]_{b_{E_2,E'_2}} &&&
                                        E'_2\star E_2}$$
We take $\psi$ to be $\{g,g'\}$.

As we pointed out in $\S$\ref{SS:H1group}, the multiplication in
$\mfZ(\Gamma,\bbG)$ makes it into a group object in the category of
2-groupoids and weak functors, and the corresponding
2-crossed-module is equivalent to $\mcK^{\leq 1}(\Gamma,\bbG)$. The
above discussion can be summarized by saying that, when $\bbG$ is
symmetric,  the 2-crossed-module  $\mcK^{\leq 1}(\Gamma,\bbG)$ is
braided and  ``symmetric''. (We use quotes because we are not
aware of a precise definition of the notion of  symmetric braided 2-crossed-module.)

The same discussion applies to the 2-group obtained by identifying
2-isomorphic arrows in  $\mfZ(\Gamma,\bbG)$.  Thus,  the
crossed-module
     $[d \:  C^{0}(\Gamma,\bbG)/B^{0}(\Gamma,\bbG)
                    \to Z^{1}(\Gamma,\bbG)]$
introduced in $\S$\ref{SS:braiding1} inherits  a symmetric braiding.
We have already  encountered  this braiding in Lemma
\ref{L:symmetric}.

\section{Functoriality of $\mfZ(\Gamma,\bbG)$}{\label{S:Functoriality}}

In this section the reader is assumed to  have some basic
familiarity with the formalism of butterflies \cite{Maps}. One
advantage of working with butterflies, as opposed to strict
morphisms of crossed-modules, is that certain calculations become
completely categorical and simple. Another advantage is that
questions regarding invariance under
equivalence of crossed-modules get automatically taken care of.

The main result of this section can be summarized by saying that
$\mfZ(\Gamma,\bbG)$ is functorial with respect to weak
$\Gamma$-equivariant morphisms $B \: \bbG \to \bbH$ of
crossed-modules (read strong $\Gamma$-equivariant butterflies).

\subsection{Strong $\Gamma$-butterflies}{\label{SS:stronggamma}}
We begin by recalling the definition of a strong butterfly
(\cite{ButterflyI}, Definition 4.1.6).

\begin{defn}{\label{D:strong}}
    A {\bf strong butterfly} $(B,s) \: \bbH \to \bbG$ consists of
    a butterfly
       $$\xymatrix@C=8pt@R=6pt@M=6pt{ H_1 \ar[rd]^{k} \ar[dd]
                          & & G_1 \ar[ld]_{i} \ar[dd] \\
                            & B \ar[ld]^{p} \ar[rd]_{r}  & \\
                                       H_0 & & G_0       }$$
    together with a set theoretic section $s \: H_0 \to B$ for $p$.
    When $\bbG$ and $\bbH$ carry a strict $\Gamma$-action, a {\bf
    $\Gamma$-butterfly} is a butterfly for which the group $B$ is
    endowed with a $\Gamma$-action such that the four maps $i$, $k$,
    $p$ and $r$ are  $\Gamma$-equivariant. A {\bf strong
    $\Gamma$-butterfly} is a $\Gamma$-butterfly whose underlying
    butterfly is strong. A morphism of strong $\Gamma$-butterflies
    is a morphism of the underlying butterflies
    ($\S$\ref{SS:butterfly}) in which the homomorphism $t \: E \to
    E'$  is $\Gamma$-equivariant. Finally, the definition of a
    2-morphism is the one which ignores the section $s$ and the
    $\Gamma$-action.
\end{defn}

\begin{rem}
    Under the correspondence between butterflies and weak morphisms,
    $\Gamma$-butterflies correspond to {\em weakly}
    $\Gamma$-equivariant
    weak morphisms.
\end{rem}

With the composition defined as in (\cite{Maps}, $\S$10.1), strong
$\Gamma$-butterflies form a bicategory which is biequivalent to the
bicategory of $\Gamma$-butterflies (via the forgetful functor
forgetting the section).

A $\Gamma$-butterfly as in Definition \ref{D:strong} gives rise to a
butterfly
           $$\xymatrix@C=8pt@R=6pt@M=6pt{ H_1 \ar[rd]^{k} \ar[dd]
                          & & G_1 \ar[ld]_{i} \ar[dd] \\
            & B\rtimes\Gamma \ar[ld]^{(p,\id)} \ar[rd]_{(r,\id)}  & \\
                  H_0\rtimes\Gamma & & G_0\rtimes\Gamma       }$$
If $B$ is strong, then this butterfly is also strong in a natural
way. This construction respects composition of (strong)
$\Gamma$-butterflies. More precisely, it gives rise to a trifunctor
from the tricategory of $\Gamma$-crossed-modules and
$\Gamma$-butterflies to the tricategory of crossed-modules and
butterflies.

\subsection{Functoriality of  $\mfZ(\Gamma,\bbG)$}{\label{SS:functoriality}}

The 2-group $\mfZ(\Gamma,\bbG)$ is functorial in the second variable
in the following sense: for a fixed $\Gamma$, $\mfZ(\Gamma,-)$ is a
trifunctor from the tricategory of $\Gamma$-crossed-modules and
strong $\Gamma$-butterflies to the tricategory of 2-groupoids. We
will not give a detailed proof of this statement. We will only
describe the effect
  $$(B,s)_* \: \mfZ(\Gamma,\bbH) \to \mfZ(\Gamma,\bbG)$$
of a strong $\Gamma$-butterfly $(B,s)$ from $\bbH$ to $\bbG$.  The
effects of morphisms and 2-arrows of  strong $\Gamma$-butterflies
are easy to describe.

Let $(B,s) \: \bbH \to \bbG$ be a strong $\Gamma$-butterfly. Let
$(E,\rho)$ be an object in  $\mfZ(\Gamma,\bbH)$, as in the diagram
        $$\xymatrix@C=8pt@R=6pt@M=6pt{
                          & & H_1 \ar[ld]_{\iota} \ar[dd]^{\partial} \\
                            & E \ar[ld]_{\pi} \ar[rd]^(0.4){\rho}  & \\
                                      \Gamma & &  H_0     }$$

We define the image under $(B,s)$ of $(E,\rho)$ in
$\mfZ(\Gamma,\bbG)$ to be the pair $(F,\lambda)$ which is defined as
follows. Consider the fiber product $K:=E\times_{\urho,H_0,p} B$.
This can be made into a group by defining the product to be
  $$(x,b)\cdot(y,c):=(xy\ ,\ b\cdot{}^{\pi(x)}\!c).$$
There is a subgroup $N$ of this group consisting of elements of the
form $\big(\iota(\al),k(\al)\big)$, $\al \in G_1$.  We define $F$ to
be $K/N$. It fits in the following diagram:
           $$\xymatrix@C=8pt@R=6pt@M=6pt{
              & & G_1 \ar[ld]_{(1,i)} \ar[dd]^{\partial} \\
                & F \ar[ld]_{\pi\circ\pr_1} \ar[rd]^(0.4){\lambda}  & \\
                                      \Gamma & &  G_0     }$$
The crossed homomorphism $\lambda \: F \to  G_0$ is given by $(x,b)
\mapsto r(b)$. It is easy to verify that $(F,\lambda)$ is an object
in $\mfZ(\Gamma,\bbG)$.

The effect  of $(B,s)$ on an arrow $(t,h) \: (E,\rho) \to
(E',\rho')$ is the pair $\big(u,rs(g)\big)$, where $u \: F \to F'$
is the homomorphism induced from the map
   $$E\times_{\urho',H_0,p} B \to E'\times_{\urho',H_0,p}, B$$
   $$(x,b) \mapsto \big(t(x)\ ,\ s(h)^{-1}\cdot b
                  \cdot{}^{\pi(x)}\!s(h)\big).$$

Finally, the effect of  $(B,s)$ on a 2-arrow $\mu \: (t,h)
\Rightarrow (t',h')$, where $\mu \in H_1$, is defined to be the
unique element $\nu \in G_1$ such that
$i(\nu)=s(g)^{-1}s(g\partial\mu)\kappa(\mu)^{-1}$.

\begin{rem}
   The functoriality of $\mfZ(\Gamma,\bbH)$ implies immediately that
   for every $\Gamma$-equivariant equivalence $f \: \bbH \to \bbG$
   of crossed-modules,   the induced bifunctor
       $$f_* \: \mfZ(\Gamma,\bbH) \to \mfZ(\Gamma,\bbG)$$
   is a   biequivalence. Therefore, the induced morphism of
   crossed-modules in groupoids
      $$f_* \: \mcK^{\leq 1}(\Gamma,\bbH)
                        \to \mcK^{\leq 1}(\Gamma,\bbG)$$
   is an equivalence (compare Proposition \ref{P:equivalence}).
   In particular, the induced maps on cohomology $H^i$, $i=-1,0,1$,
   are isomorphisms.
\end{rem}

\subsection{In the presence of a braiding}{\label{SS:braidedfunctoriality}}

\begin{defn}{\label{D:braidedbutterfly}}
  A  $\Gamma$-butterfly
       $$\xymatrix@C=8pt@R=6pt@M=6pt{ H_1 \ar[rd]^{k} \ar[dd]
                          & & G_1 \ar[ld]_{i} \ar[dd] \\
                            & B \ar[ld]^{p} \ar[rd]_{r}  & \\
                                       H_0 & & G_0       }$$
  is {\bf braided} if it satisfies the identity
  $$k\{p(b),p(c)\}_{\bbH}\cdot
  i\{r(b),r(c)\}_{\bbG}=b^{-1}c^{-1}bc$$ for every $b,c \in B$. A
  {\bf strong braided  $\Gamma$-butterfly} is a braided
  $\Gamma$-butterfly together with a set theoretic section $s \: H_0
  \to E$ for $p$. Morphisms and 2-morphisms of strong braided
  $\Gamma$-butterflies are defined to be the ones of  the underlying
  $\Gamma$-butterflies
\end{defn}

If $\bbG$ and $\bbH$ are endowed with a $\Gamma$-equivariant
braiding and $B$ is a braided $\Gamma$-butterfly in the sense of
Definition \ref{D:braidedbutterfly}, then the bifunctor
       $$(B,s)_* \: \mfZ(\Gamma,\bbH) \to \mfZ(\Gamma,\bbG)$$
is monoidal. The monoidal structure on this functor is given by the
natural isomorphisms
  $$F_{E,E'}\: B_*(E)\star B_*(E') \to B_*(E\star E'),$$
  $$\big((x,b),(y,c)\big) \mapsto (x,y,bc).$$
Here we have abbreviated $(B,s)_*$ to $B_*$ and $(E,\rho)$ to $E$.
(The proof that this map is a group homomorphism
is quite nontrivial and involves some lengthy calculations.) Also, given
two arrows $(t,h) \: (E_1,\rho_1) \to (E_2,\rho_2)$ and
 $(t',h') \: (E'_1,\rho'_1) \to (E'_2,\rho'_2)$ in $\mfZ(\Gamma,\bbH)$,
 we have the following commutative 2-cell in $\mfZ(\Gamma,\bbH)$:
        $$\xymatrix@=3pt@M=8pt{ B_*(E_1)\star B_*(E_1')
                 \ar[ddd]_{B_*(t',h')\star B_*(t,h)}
               \ar[rrr]^{F_{E_1,E'_1}} &&&
          B_*(E_1'\star E_1) \ar[ddd]^{B_*\big((t,h)\star(t',h')\big)} \\
                        &&& \ar@{=>} [lldd]_{\vep(h,h')}  \\
                            &&&     \\
                     B_*(E_2)\star B_*(E_2')
                      \ar[rrr]_{F_{E_2,E'_2}} &&&  B_*(E'_2\star E_2)}$$
where $\vep(h,h')\in G_1$ is the unique element in $G_1$ satisfying
the identity $i\vep(h,h')=s(hh')^{-1}s(h)s(h')$.

\begin{rem}
  It can be shown that, for a fixed $\Gamma$, $\mfZ(\Gamma,-)$ is a
  trifunctor from the tricategory of braided $\Gamma$-crossed-modules
  and braided strong $\Gamma$-butterflies to the tricategory of
  monoidal 2-groupoids. We will not prove this here.
\end{rem}

It follows from the above discussion that if $(B,s) \: \bbH \to
\bbG$ is a braided strong $\Gamma$-butterfly, then  the induced weak
morphism
       $$(B,s)_* \: \mcZ(\Gamma,\bbH) \to \mcZ(\Gamma,\bbG)$$
of 2-groups is braided. This implies that the induced map
    $$\mcK^{\leq 1}(\Gamma,\bbH) \to \mcK^{\leq 1}(\Gamma,\bbG)$$
is a morphism of 2-crossed-modules.   In particular, the induced map
    $$H^1(\Gamma,\bbH) \to H^1(\Gamma,\bbG)$$
is a group homomorphism.

\section{Everything over a  Grothendieck site}{\label{S:Site}}

The discussion of $\S$\ref{S:Functoriality} is valid over any
Grothendieck site, but some changes need to be made in the
definition of  $\mcZ(\Gamma,\bbG)$. We discuss this in this section
and prepare the ground to compare our definition of $H^i$ with the
standard one in terms of gerbes.

Let $X$ be a fixed Grothendieck site. By a group we mean a sheaf of
groups over $X$.  A short exact sequence means a short exact
sequence of sheaves of groups.

Let $\bbG=[G_1\to G_0]$ be a crossed-module over $X$, and $\Gamma$ a
group over $X$ acting strictly on $\bbG$. We would like to define
the analog of the 2-groupoid $\mfZ(\Gamma,\bbG)$. The definition is
more or less the same as in the discrete case, with one slight
change in the definition of arrows (hence, also of 2-arrows).
For this reason, we will use a different notation $\mfZ'(\Gamma,\bbG)$
for it.

The crossed-module $\bbG=[\partial \: G_1\to G_0]$ gives rise to a
quotient stack $[G_0/G_1]$ where $G_1$ acts on $G_0$ by right
multiplication via $\partial$.  That is, $[G_0/G_1]$ is the quotient
stack of the transformation groupoid $[G_0\ltimes G_1\sst{} G_0]$.
The latter is a strict group object  in the category of groupoids.
Therefore, the quotient stack $[G_0/G_1]$ is in fact a group stack.
We denote this group stack by $\mcG$.

\subsection{A provisional definition in terms of group stacks}
{\label{SS:preliminary}}

We include this subsection just to motivate the definition of
$\mfZ'(\Gamma,\bbG)$ that will be given in $\S$\ref{SS:global}.
Using the idea  discussed in Remark \ref{R:reformulation}, we
introduce a closely related (and naturally biequivalent) 2-groupoid
which is defined in terms of group stacks. This 2-groupoid, though
conceptually much simpler, is not very explicit. In
$\S$\ref{SS:global} we use the results of \cite{ButterflyI} to
translate this definition to the language of crossed-modules.

Let $\mcG$ be a group stack (for example, the quotient stack of
$\bbG$) with an action of a group $\Gamma$.  The associated
2-groupoid is defined as follows.

An object in this 2-groupoid is a (weak) morphism of group stacks $r
\: \Gamma \to \mcG\rtimes\Gamma$ such that $\pr_2\circ
r=\id_{\Gamma}$.

A morphism $(t,g) \: r \to r'$  in this 2-groupoid consists of  a
global section $g$ of $\mcG$ and a monoidal transformation $t \: r
\to gr'g^{-1}$, where  $gr'g^{-1}$ is the morphism $r'$ composed
with the conjugation by $g$  automorphism of $\mcG$.  The
composition of two morphisms $(t,g)$ and $(t',g')$ is defined to be
$\big(t(gt'g^{-1}), gg'\big)$.

A 2-arrow $\mu \: (t,g) \Rightarrow (t',g')$ is a transformation
$\mu \: g \to g'$ which intertwines $t$ and $t'$.

\subsection{The 2-groupoid $\mfZ'(\Gamma,\bbG)$}{\label{SS:global}}

Thanks to the equivalence of butterflies and weak morphisms of group
stacks \cite{ButterflyI}, we can translate the definition  given in
$\S$\ref{SS:preliminary} and find a more convenient definition of
$\mfZ'(\Gamma,\bbG)$ along the lines of $\S$\ref{S:CohviaButterfly}.

\medskip

{\noindent}{\bf Some notation.} We use the notation $X\ox{G}Y$ for
the contracted product of two sets $X$ and $Y$ with an action of a
group $G$. Breen (and also \cite{ButterflyI}) uses the  notation
$X\overset{G}{\wedge}Y$. If $X$ and $Y$ are over a third set $Z$ and
the $G$-actions are fiberwise, we denote by   $X\oux{Z}{G}Y$ the
subset in $X\ox{G}Y$ consisting of those pairs $(x,y)$ such that $x$
and $y$ map to the same element in $Z$.

\medskip

\noindent{\em Objects of $\mfZ'(\Gamma,\bbG)$.}
The objects of $\mfZ'(\Gamma,\bbG)$
turn out to be exactly the same as before. Namely, they are diagrams
        $$\xymatrix@C=8pt@R=6pt@M=6pt{
                & & G_1 \ar[ld]_{\iota} \ar[dd]^{\partial} \\
                  & E  \ar[ld]_{\pi} \ar[rd]^(0.4){\rho}  & \\
                                      \Gamma & &  G_0       }$$
of sheaves of groups over $X$ such that the diagonal sequence is
short exact and $\rho$ is a crossed-homomorphism intertwining the
conjugation action of $E$ on $G_1$ with the crossed-module action of
$G_0$ on $G_1$ (see $\S$\ref{SS:H1new}).

\medskip

\noindent{\em Arrows of $\mfZ'(\Gamma,\bbG)$.} An arrow in
$\mfZ'(\Gamma,\bbG)$ from $(E,\rho)$ to $(E',\rho')$ is a pair
$(t,g)$ where $g$ and $t$ are as follows. The $g$ here is a pair
$(P,\vf)$, where $P$ is a right $G_1$-torsor on $X$ and $\vf \: P
\to G_0$ is a $G_1$-equivariant morphism of sheaves. Here $G_1$ acts
on $G_0$ by right multiplication via $\partial$. The $t$ is an
isomorphism $E \to {}^g\!E'$ of sheaves of groups making the
following diagram commute
          $$\xymatrix@C=30pt@R=6pt@M=4pt{
               & & G_1 \ar[ld]_{{}^g\!\iota'}
                    \ar[lddd]^(0.4){\iota}  |!{[ld];[dddd]}\hole
                                \ar[dddd]^{\partial} \\
                      &   {}^g\!E'   \ar[lddd]_{{}^g\!\pi'}
                         \ar[rddd]^{{}^g\!\rho'} & \\
                            & &  \\
             & E  \ar[uu]^(0.4)t \ar[ld]^{\pi} \ar[rd]_(0.4){\rho}  & \\
                                      \Gamma & &  G_0       }$$

Here, ${}^g\!E':=P\ox{G_1}E'$ is the contracted product of  $P$ and
$E'$, where $G_1$ acts on $E'$ by right conjugation. The map
${}^g\!\pi'\: P\ox{G_1}E' \to \Gamma$
is $\pi'\circ\pr_2$. The map ${}^g\!\rho'$ is  defined by
   $${}^g\!\rho' \: P\ox{G_1}E' \to G_0$$
    $$ (u,x)\mapsto  \vf(u)\cdot\rho'(x)\cdot
                {}^{\pi'(x)}\!\vf(u)^{-1},$$
and ${}^g\!\iota'$ is defined by
       $$ {}^g\!\iota' \: G_1 \to P\ox{G_1}E'$$
    $$\al \mapsto \big(u,\iota'(\al^{\vf(u)})\big),$$
where $u \in P$ is randomly chosen; it is easy to see that the pair
$\big(u,\iota'(\al^{\vf(u)})\big)$, viewed as
an element in  $P\ox{G_1}E'$, is independent of $u$.

\begin{rem}
 The object $({}^g\!E,{}^g\!\rho)$ should be regarded as the left
 conjugate of $(E,\rho)$ under the action of $g$. Note that, by
 definition of the quotient stack,
 $g=(P,\vf)$ is a global section of $\mcG=[G_0/G_1]$.
\end{rem}

The composition of two arrows $(t,g) \: (E,\rho) \to (E',\rho')$ and
$(t',g') \: (E',\rho') \to (E'',\rho'')$ is defined to be
$(t'',g'')$, where $g''$ and $t''$ are defined as follows. First we
define $g''$. Let $g=(P,\vf)$ and $g'=(P',\vf')$. Make $P'$ into a
left $G_1$-torsor (indeed, a bitorsor) by setting
  $${}^{\alpha}\!u:=u^{\alpha^{\vf(u)}} , \  \alpha \in G_1,  u \in P'.$$
Form the contracted product   $P\ox{G_1}P'$, where now $P'$ is
viewed as a left $G_1$-torsor. It inherits
a right $G_1$-torsor structure from $P'$. Define
$\vf\ox{G_1}\vf'$ by the rule
    $$\vf\ox{G_1}\vf' \: P\ox{G_1}P' \to G_0, \ \
      (u,v) \mapsto \vf(u)\vf'(v).$$
We define $g''$ to be the pair $(P\ox{G_1}P',\vf\ox{G_1}\vf')$.
(Remark that if we view $g$ and $g'$ as global sections of the group
stack $\mcG=[G_0/G_1]$, then $g''$ corresponds to the product
$gg'$).

The homomorphism $t''$ is defined to be
    $$t'' \: E \to P\ox{G_1}P'\ox{G_1}E''$$
     $$t'':=(P\ox{G_1}t') \circ t.$$

\medskip

\noindent{\em 2-arrows of $\mfZ'(\Gamma,\bbG)$.} A 2-arrow $(t,g)
\Rightarrow (t',g')$ in $\mfZ'(\Gamma,\bbG)$ is an isomorphism $\mu
\: g \to g'$ such that the  diagram
     $$\xymatrix@C=2pt@R=14pt@M=8pt{  {}^g\!E
                           \ar[rr]^{\mu\ox{G_1} E} \ar[rd]_{t}  & &
             {}^{g'}\!E \ar[ld]^{t'} \\
                         &   E & }$$
commutes. Here, by an isomorphism $\mu\:g \to g'$ we mean an
isomorphism $P_g \to P_{g'}$ of $G_1$-torsors, which we denote again
by $\mu$,  making the diagram
        $$\xymatrix@C=0pt@R=12pt@M=8pt{  P_g \ar[rr]^{\mu}
                      \ar[rd]_{\vf}  & &  P_{g'} \ar[ld]^{\vf'} \\
         &   G_0 & }$$
commute.

\begin{rem}
  In contrast with $\mfZ(\Gamma,\bbG)$
  which is a 2-groupoid, $\mfZ'(\Gamma,\bbG)$ is a bigroupoid.
\end{rem}

\subsection{Functoriality of $\mfZ'(\Gamma,\bbG)$}
{\label{SS:globalfunctoriality}}

The bigroupoid $\mfZ'(\Gamma,\bbG)$ is in some sense more natural
than $\mfZ(\Gamma,\bbG)$, because it is actually  functorial with
respect to $\Gamma$-butterflies. That is, we do not need strong
butterflies (Definition \ref{D:strong}) in order  to define
pushforwards.

Let $B \: \bbH \to \bbG$ be a $\Gamma$-butterfly
             $$\xymatrix@C=8pt@R=6pt@M=6pt{ H_1 \ar[rd]^{k} \ar[dd]
                          & & G_1 \ar[ld]_{i} \ar[dd] \\
                            & B \ar[ld]^{p} \ar[rd]_{r}  & \\
                                       H_0 & & G_0       }$$
The bifunctor $B_* \: \mfZ'(\Gamma,\bbH) \to \mfZ'(\Gamma,\bbG)$ is
defined as follows.

\medskip
\noindent{\em Effect of $B_*$ on objects.}
 Let $(E,\rho)$ be an object in
 $\mfZ'(\Gamma,\bbH)$. The effect of $B_*$ is given by
        $$\xymatrix@C=8pt@R=6pt@M=6pt{
             & & H_1 \ar[ld]_{\iota} \ar[dd]^{\partial}    &&&&
             & & G_1 \ar[ld]  \ar[dd]^{\partial} \\
             & E \ar[ld]_{\pi} \ar[rd]^(0.4){\rho}  &
                        &&\overset{B_*}{\mapsto} &&
             & F \ar[ld] \ar[rd]^(0.4){\lambda}  &  \\
                                      \Gamma & &  H_0      &&&&
                                      \Gamma & &  G_0         }$$
where $B_*(E,\rho)$ is defined exactly as in
$\S$\ref{SS:functoriality}. Namely,
it is equal to $(F,\lambda)$ with
  $$F:=E\oux{H_0}{H_1}B \ \ \text{and} \ \
            \lambda \: F \to  G_0, \ (x,b)  \mapsto r(b).$$
Here, $ G_1$ acts on each component by right multiplication.

\medskip
\noindent{\em Effect of $B_*$ on arrows.} Let  $(t,h) \: (E,\rho)
\to (E',\rho')$ be an arrow in $\mfZ'(\Gamma,\bbH)$. Here,  $h$ is
equal to $(P,\vf)$, where $P$ is an $H_1$-torsor and  $\vf \: P \to
H_0$ is  an $H_1$-equivariant map, and $t \: E \to P\ox{H_1}E'$ is a
homomorphism. We define $B_*(t,h)$  to be $(s,g)$, where $g$ and $s$
are defined as follows.

Consider $Q:=P\oux{H_0}{H_1}B$, where $H_1$ acts on $B$ by right
multiplication via $k$. Since the images of $k$ and $i$ in $B$
commute, the right multiplication action of $G_1$ on $B$ via $i$
gives rise to a right action of $G_1$ on $Q$. It is easy that this
makes $Q$ into a right $G_1$-torsor. We have a $G_1$-equivariant map
  $$\chi \:  Q=P\oux{H_0}{H_1}B \to G_0$$
    $$(x,b) \mapsto r(b).$$
We define $g$ to be $(Q,\chi)$. The homomorphism $s \: F \to
{}^g\!F'$ is defined to be the composition
  {\small   $$\xymatrix{
         F= E\oux{H_0}{H_1}B     \ar[r]^{t\oux{H_0}{H_1}B}
               &
          (P\ox{H_1}E')\oux{H_0}{H_1}B   \ar[r]^(0.35){\eta^{-1}} &
          (P\oux{H_0}{H_1}B)\ox{G_1}(E'\oux{H_0}{H_1}B)
                     =  Q\ox{G_1}F'= {}^g\!F'
                                 }$$  }
For the convenience of the reader, let us clarify all the actions
appearing in the above expression, as well as define
the isomorphism  $\eta$.

In $(P\ox{H_1}E')\oux{H_0}{H_1}B$, the action of the first $H_1$ on
$E'$ is by right conjugation, and the action of the second $H_1$ on
$B$ is by right multiplication. The action of the second $H_1$ on
$P\ox{H_1}E'$ is by right multiplication via ${}^g\!\iota'$.
That is,   $(u,x)$ acted on by  $\al \in H_1$ is equal to
$(u,x\al^{\vf(u)})$.

In $ (P\oux{H_0}{H_1}B)\ox{G_1}(E'\oux{H_0}{H_1}B) $ all actions are
by right multiplication, except for the action of $G_1$ on the last
$B$ component which is by right conjugation.

Finally, the isomorphism $\eta$ is defined by
   $$\eta \:  (P\oux{H_0}{H_1}B)\ox{G_1}(E'\oux{H_0}{H_1}B)  \ \to \
             (P\ox{H_1}E')\oux{H_0}{H_1}B$$
            $$(u,b,y,c) \ \mapsto \  (u,y,bcb^{-1}).$$
We leave it to the reader to verify that this is indeed an isomorphism
of groups.

\medskip
\noindent{\em Effect of $B_*$ on 2-arrows.}
This is defined in the obvious way.

\subsection{Comparing $\mfZ'(\Gamma,\bbG)$ and $\mfZ(\Gamma,\bbG)$}
{\label{SS:comparison}}

Instead of defining  $\mfZ'(\Gamma,\bbG)$ as in $\S$\ref{SS:global},
we could  have imitated the definition of $\mfZ(\Gamma,\bbG)$ given
in $\S$\ref{SS:H1new}. We argue that this would not have been the
correct definition. Let us analyze what goes wrong with this naive
definition. There is a natural bifunctor
  $$\Psi \: \mfZ(\Gamma,\bbG) \to \mfZ'(\Gamma,\bbG)$$
which is the identity on objects and is fully faithful on hom
groupoids. This functor, however, misses many arrows in
$\mfZ'(\Gamma,\bbG)$. This is essentially because not every global
section of the quotient stack $[G_0/G_1]$ lifts to a global section
of $G_0$. Let us spell this out in more detail.

The functor $\Psi$  sends an arrow $(t,g)$ in  $\mfZ(\Gamma,\bbG)$,
where $g \in G_0$ and $t \: E \to E'$ is a group homomorphism (with
certain properties),  to  a pair $(\hat{g},\hat{t})$ in which
$\hat{g}$ is the pair $(G_1,\vf)$ with $G_1$ the trivial
$G_1$-torsor and $\vf \: G_1 \to G_0$ given by $\al \mapsto
g\partial(\al)$. It follows that  if an arrow in
$\mfZ'(\Gamma,\bbG)$ is in the image of $\Psi$, or is 2-isomorphic
to such an arrow, then its corresponding $G_1$-torsor $P$ is
trivial. The converse is also easily seen to be true.

\begin{prop}
   There is a natural bifunctor
      $$\Psi \: \mfZ(\Gamma,\bbG) \to \mfZ'(\Gamma,\bbG)$$
   which is the identity on objects and is fully faithful on hom
   groupoids. If $H^1(X,G_1)$ is trivial, then $\Psi$ is a
   biequivalence. In particular, in the case where everything is
   discrete (i.e., $X$ is a point),  $\Psi$ is a biequivalence.
\end{prop}

To end this subsection, let us also recall two other differences
between $\mfZ(\Gamma,\bbG)$ and $\mfZ'(\Gamma,\bbG)$. The former is
a 2-groupoid and it is functorial only with respect to {\em strong}
$\Gamma$-butterflies. The latter is a bigroupoid
and is functorial with respect to all $\Gamma$-butterflies.

\subsection{Continuous, differentiable, algebraic, etc., settings}
{\label{SS:diff}}

The cocycle approach to  cohomology discussed in
$\S$\ref{S:0}-\ref{S:B} has the disadvantage that it is only
appropriate in the discrete setting. For instance, in the case where
$\Gamma$ and $\bbG$ are Lie,  both the differentiable and discrete
cocycles give the wrong cohomologies in general.

The butterfly approach, however, always gives the correct answer.
Let us elaborate this a little bit. For example, suppose that $M$ is
a manifold, $\Gamma$ is a Lie group bundle over $M$, and
$\bbG=[G_1\to G_0]$ a bundle of Lie crossed-modules. In this case,
an element in  $H^1(\Gamma,\bbG)$ is a diagram
   $$\xymatrix@C=8pt@R=6pt@M=6pt{
                          & & G_1 \ar[ld]_{\iota} \ar[dd] \\
                            & E \ar[ld]_{\pi} \ar[rd]^{\rho}  & \\
                                       \Gamma & & G_0       }$$
as in $\S$\ref{SS:global} in which $E$ is a Lie group bundle over
$M$, the diagonal sequence is short exact in the category of Lie
group bundles, and the map $\rho$ is differentiable. Two such
diagrams $(E,\rho)$ and $(E',\rho')$ give rise to the same
cohomology class in  $H^1(\Gamma,\bbG)$ if and only if there exists
a principal $G_1$-bundle $P$ over $M$, a $G_1$-equivariant
differentiable map of bundles
$\vf \: E \to  G_0$,  and an isomorphism of Lie group bundles
            $$f \: P\ox{G_1}E' \to E$$
such that:
 \begin{itemize}
    \item for every $u \in P$ and $y \in E'$,
           $\rho f(u,y)\cdot {}^{\pi'(y)}
                 \!\vf(u)=\vf(u)\cdot\rho'(y)$,
    \item for every $u \in P$ and $y \in E'$, $\pi f(u,y)=\pi'(y)$,
    \item for every $u \in P$ and $\al \in G_1$,
              $f(u,\iota'(\al^{\vf(u)}))=\iota(\al)$.
 \end{itemize}

Notice, in particular, that in the case where $M$ is a point, the
Lie group $E$ and the extension
   $$ 1 \to G_1 \to E \to \Gamma \to 1$$
are uniquely determined (up to isomorphism) by the the given element
in $H^1(\Gamma,\bbG)$ and can be thought as invariants of the given
cohomology class.

The same discussion is valid in the algebraic setting (where $G$ is
a group scheme, or an algebraic group, and $\bbG$ is a
crossed-module in group schemes, or algebraic groups),
or in the topological setting, etc.

\section{$H^i$ and gerbes}{\label{S:Gerbes}}

In this section, we give an interpretation of the 2-groupoid
$\mfZ'(\Gamma,\bbH)$ in terms of gerbes over the classifying stack
$B\Gamma$, and clarify the relation between our definition
of $H^i$ and the standard one in terms of gerbes.

Our set up is as follows. We fix a Grothendieck site $X$. When
working over the site $X$, by a group we  mean a sheaf of groups on
$X$, and by a crossed-module we mean a
crossed-module in sheaves of groups.

Given a sheaf of groups $\Gamma$ over $X$, we denote the classifying
stack of $\Gamma$ by $B\Gamma:=[\Gamma\backslash X]$. We sometimes
use the same notation for the Grothendieck site
$(X\!\downarrow\!B\Gamma)$ of objects in $X$ over $B\Gamma$.

Recall that to a  crossed-module $\bbG=[\partial \: G_1\to G_0]$
over $X$ we can associate a group stack $\mcG$ which is, by
definition, the quotient stack of the transformation groupoid
$[G_0\ltimes G_1\sst{} G_0]$. Note that the latter is a strict group
object  in the category of groupoids.

Our notational convention is that whenever we use the notation
$[G_0/G_1]$, we simply mean the quotient stack without the group
structure. When we want to take into account the group structure, we
use $\mcG$. For example, we will be considering
$[G_0/G_1]$ as a trivial right $\mcG$-torsor.

\subsection{Cohomology via gerbes}{\label{SS:cohomologygerbes}}

It is well-known that, for every group stack $\mcG$ over a
Grothendieck site $X$, $H^{i}(X,\mcG)$, $i=-1,0,1$ are defined as
follows:
\begin{itemize}
  \item $H^{-1}(X,\mcG)$ is the group of self-equivalences of
                the identity section of $\mcG$;
      this is an abelian group.
  \item $H^{0}(X,\mcG)$ is the group of global sections of
                $\mcG$ modulo transformation;
     this is a group,  not necessarily abelian.
  \item  $H^{1}(X,\mcG)$ is the set of isomorphism classes of (right)
      $\mcG$-torsors over $X$; this is a pointed set.
\end{itemize}

A  $\Gamma$-crossed-module $\bbG$ gives rise to a crossed-module
$\bbGg$, and the corresponding group stack $\mcGg$, on the
classifying stack $B\Gamma$. In the case where $X$ is a point, it is
straightforward (but rather tedious) to see that we have natural
isomorphisms
   $$H^i(\Gamma,\bbG)\cong H^i(B\Gamma,\mcGg), \ \ i=-1,0,1.$$
In fact, our definitions of $H^i(\Gamma,\bbG)$ given in
$\S$\ref{S:0} and $\S$\ref{S:1} were obtained by translating the
definition of $H^i(B\Gamma,\mcGg)$ to the cocycle language. (The
idea is to write down the descent data for a $\mcGg$-torsor on
$B\Gamma$ and see that we obtain the cocycles of  $\S$\ref{S:0} and
$\S$\ref{S:1}.)

The right $\mcGg$-torsors over $B\Gamma$ form a strict 2-groupoid.
The morphisms of this groupoid  are morphisms of $\mcGg$-torsors,
and the 2-arrows of it are transformations. Let $\mfZ(\Gamma,\mcG)$
be the full sub 2-groupoid of this 2-groupoid consisting of those
$\mcGg$-torsors which become isomorphic to the trivial $\mcG$-torsor
when pulled back to $X$ via the quotient map $X \to B\Gamma$.
(We do not fix the trivialization.)

\begin{prop}{\label{P:biequiv}}
 Let $\mfZ(\Gamma,\bbG)$ be as above and $\mfZ'(\Gamma,\bbG)$
 as in $\S$\ref{SS:global}. Then,  there is a biequivalence
   $$\Upsilon\: \mfZ'(\Gamma,\bbG) \to \mfZ(\Gamma,\mcG)$$
 which is natural up to higher coherences.
\end{prop}

\begin{proof}
 We give an outline of the construction of this biequivalence.

\medskip
\noindent{\em Effect of $\Upsilon$ on objects.}
Let  $(E,\rho)$ be an object in   $\mfZ'(\Gamma,\bbG)$, as in the
diagram
        $$\xymatrix@C=8pt@R=6pt@M=6pt{
                 & & G_1 \ar[ld]_{\iota} \ar[dd]^{\partial} \\
                 & E  \ar[ld]_{\pi} \ar[rd]^(0.4){\rho}  & \\
                                      \Gamma & &  G_0       }$$
To this we want to associate a right $\mcGg$-torsor over $B\Gamma$.
Think of $[\iota \: G_1 \to E]$ as a crossed-module (via the
conjugation action of $E$ on $G_1$), and let $\tilde{\Gamma} :=[E/G_1]$ be
the corresponding group stack. That is, the underlying stack of
$\tilde{\Gamma} $ is the quotient stack of the groupoid 
$[E\ltimes G_1\sst{}E]$. As we pointed out at the beginning of this section, 
the latter is a strict group object in the category of groupoids. There is a
natural equivalence of group stacks $\phi \: \tilde{\Gamma} \to
\Gamma$ induced by $\pi$.

The map $\phi$ provides us a left action of $\tilde{\Gamma}$ on the
group stack $\mcG$ via that of $\Gamma$. We will show that there is
also a natural action of $\tilde{\Gamma}$ on the stack $[G_0/G_1]$
which makes the right $\mcG$-torsor structure of $[G_0/G_1]$
$\tilde{\Gamma}$-equivariant. After {\em choosing} an inverse for
$\phi$, this gives rise to an action of $\Gamma$ on the trivial
$\mcG$-torsor $[G_0/G_1]$. Passing  to $\Gamma$-quotients, we obtain
a $\mcGg$-torsor $\mathcal{P}$ on the classifying stack $B\Gamma$.

Let us now spell out the action of $\tilde{\Gamma}$ on the stack
$[G_0/G_1]$. We do this on the groupoid level. That is, we give a
left  action of $[E\ltimes G_1\sst{} E]$, viewed as a group object
in groupoids, on the groupoid $[G_0\ltimes G_1\sst{} G_0]$. To do
so, we give an automorphism $F_x$ from $[G_0\ltimes G_1\sst{} G_0]$ to
itself for every $x \in E$. Also, for every arrow $(x,\be)$ between
the objects $x$ and $y=x\iota(\be)$ in $[E\ltimes G_1\sst{} E]$, we
give a transformation $T_{(x,\be)} \: F_x \Rightarrow F_y$.

The effect of the automorphism $F_x$ on an object $g \in G_0$ of the
groupoid $[G_0\ltimes G_1\sst{} G_0]$ is given by
   $$g \ \mapsto \ \rho(x)\cdot {}^{\pi(x)}\!g.$$
Its effect on an arrow $(g,\al) \in G_0\ltimes G_1$ is given by
    $$(g,\al) \ \mapsto \ \big(\rho(x)\cdot {}^{\pi(x)}\!g\
             ,\ {}^{\pi(x)}\!\al\big). $$
The transformation  $T_{(x,\be)} \: F_x \Rightarrow F_y$ is defined
by
    $$T_{(x,\be)}(g):=\big(\rho(x)\cdot {}^{\pi(x)}\!g\ ,\
                       \be^{({}^{\pi(x)}\!g)}\big).$$
Here, $g \in G_0$ is viewed as on object and $\big(\rho(x)\cdot
{}^{\pi(x)}\!g\ ,\ \be^{({}^{\pi(x)}\!g)}\big) \in G_0\ltimes G_1$
as an arrow in the groupoid $[G_0\ltimes G_1\sst{} G_0]$.

\medskip
\noindent{\em Effect of $\Upsilon$ on arrows and 2-arrows.} Let
$(t,g)\: (E,\rho) \to (E',\rho')$ be an arrow in
$\mfZ'(\Gamma,\bbG)$. Let $\bar{g}$ be the global section of
$[G_0/G_1]$ over $X$ corresponding to $g=(P,\vf)$. It can be checked
that left multiplication on $[G_0/G_1]$ by $\bar{g}$ is
$\Gamma$-equivariant and it respects the right $\mcG$-torsor
structures.  (This is perhaps easiest to see by trivializing the
$G_1$-torsor $P$ over some open cover of $X$ and then showing that
the $\Gamma$-equivariance data are compatible along the
intersections of the open sets.) After passing to the
$\Gamma$-quotients, we obtain an equivalence of $\mcGg$-torsors
$\mathcal{P} \to \mathcal{P'}$ over $B\Gamma$. This defines the
effect of $\Upsilon$ on morphisms.

The definition of the effect of $\Upsilon$ on 2-arrows is
straightforward.

\medskip
\noindent{\em The inverse of $\Upsilon$.} Denote $B\Gamma$ by $\mcY$
for simplicity. Let $\mcP$ be a $\mcGg$-torsor over $\mcY$ which
becomes trivial after pulling back along the quotient map $q \: X
\to \mcY$. {\em Choose} a trivialization, that is, a map $f \: X \to
\mcP$ relative to $\mcY$. Let $\delta \: \mcP\times_{\mcY}\mcP \to
\mcGg$ be ``the'' difference map. That is, $\delta$ is a morphism
such that
  $$\mcP\times_{\mcY} \mcP \llra{(\pr_1,\delta)}
   \mcP\times_{\mcY} \mcGg$$
becomes an inverse to
      $$\mcP\times_{\mcY} \mcGg \llra{(\pr_1,\mu)}
                 \mcP\times_{\mcY} \mcP,$$
where $\mu$ stands  for the action. (Note that we have to make a
{\em choice} of the inverse, and $\delta$ depend on this choice.)

Consider the morphism
    $$\xymatrix@C=30pt{ X \times_{\mcY} X
       \ar[r]^{f\times_{\mcY} f} &
          \mcP\times_{\mcY}\mcP \ar[r]^{\delta} &
                         \mcGg}.$$
Observe that this morphism is over $\mcY$ and that $X \times_{\mcY}
X$ is naturally equivalent to $\Gamma$. Thus, we obtain a morphism
$\Gamma \to \mcGg$ fitting in a 2-cartesian diagram
      $$\xymatrix{ \Gamma \ar[r] \ar[d] & \mcGg \ar[d] \\
           X \ar[r]_q & \mcY }$$
Since the pullback of $\mcGg$ along $q$ is naturally equivalent to
$\mcG$, we obtain a morphism $\rho \: \Gamma \to \mcG$. This map can
be checked to be a crossed-homomorphism.

Now, we follow the argument of (\cite{ButterflyI}, $\S$4.2.4)
and set
   $$E:=\Gamma \times_{\bar{\rho} ,\mcG, q} G_0,$$
where $q \: G_0 \to \mcG$ is the quotient map. By
(\cite{ButterflyI}, $\S$4.2.5), $E$ fits in  a diagram
           $$\xymatrix@C=8pt@R=6pt@M=6pt{
                & & G_1 \ar[ld]_{\iota} \ar[dd]^{\partial} \\
                & E  \ar[ld]_{\pr_1} \ar[rd]^(0.4){\pr_2}  & \\
                                      \Gamma & &  G_0       }$$
with the desired properties.  This defines the effect of
$\Upsilon^{-1}$ on objects.

From the above construction it is clear how to define the effect of
$\Upsilon^{-1}$ on arrows and 2-arrows.
\end{proof}
\subsection{In the presence of a braiding}{\label{SS:braidinggerbes}}

In this subsection, we show that if $\bbG$ is endowed with a
$\Gamma$-equivariant braiding, then there is a product on
$\mfZ(\Gamma,\mcG)$ which makes it into a (weak)  group object in
the category of 2-groupoids. Our construction was conceived in a
discussion with E.~Aldrovandi and relies on the tools developed in
(\cite{ButterflyI}, $\S$7), to which we refer the reader for more
details.

Suppose that $\bbG$ is equipped with a $\Gamma$-equivariant
braiding. In this case, the butterfly of Example \ref{E:braided}
becomes $\Gamma$-equivariant. Therefore, we have a butterfly $\bbGg
\times \bbGg \to \bbGg$ over $B\Gamma$. This in turn gives rise to a
morphism $m \: \mcGg \times \mcGg \to \mcGg$ of group stacks over
$B\Gamma$. It follows that with this multiplication $\mcGg$ is a
group object in the category of group stacks over $B\Gamma$.

We can use the morphism $m$ to define a multiplication on
$\mfZ(\Gamma,\mcG)$ as follows. Let $\mathcal{P}_1$ and
$\mathcal{P}_2$ be $\mcGg$-torsors over $B\Gamma$. Then,
$\mathcal{P}_1\times\mathcal{P}_2$ is a $\mcGg\times\mcGg$-torsor.
The `extension of structure group' functor for the map $m \: \mcGg
\times \mcGg \to \mcGg$ applied to the $\mcGg\times\mcGg$-torsor
$\mathcal{P}_1\times\mathcal{P}_2$ gives a $\mcGg$-torsors
$\mathcal{P}_1\cdot\mathcal{P}_2$. This is the desired product of
$\mathcal{P}_1$ and $\mathcal{P}_2$. More precisely,
     $$\mathcal{P}_1\cdot\mathcal{P}_2:=
             (\mathcal{P}_1\times\mathcal{P}_2)
                           \ox{\mcGg\times\mcGg}\mcGg,$$
where the $\mcGg$ on the right is made into a left
$\mcGg\times\mcGg$-torsor via $m$.

The same construction can be used to define the product of morphisms
and 2-arrows of $\mfZ(\Gamma,\mcG)$.

In the case where the braiding on $\bbG$ is symmetric, $\bbGg$
becomes a symmetric braided crossed-module over $B\Gamma$ and the
multiplication $m \: \mcGg \times \mcGg \to \mcGg$ becomes braided
(\cite{ButterflyI}, $\S$7.2). That is, $m$ and $m\circ \tau$, where
$\tau \: \mcGg \times \mcGg \to \mcGg \times \mcGg$ is the switch
map, become isomorphic via a natural isomorphism satisfying the
well-known coherence relations. This implies that the product on
$\mfZ(\Gamma,\mcG)$ is braided. This braiding is compatible with the
braiding of $\mfZ(\Gamma,\bbG)$ under the equivalence of Proposition
\ref{P:biequiv}.

\section{Cohomology long exact sequence}{\label{S:Long}}

In this section, we show that to any short exact sequence of
$\Gamma$-crossed-modules and $\Gamma$-butterflies
one can associate a long exact sequence in cohomology
(Proposition \ref{P:longexact}).

\subsection{Short exact sequences of butterflies}{\label{SS:short}}

Let $\bbK \llra{C} \bbH$ and $\bbH \llra{B} \bbG$ be butterflies. We say that
      $$1 \to \bbK \llra{C} \bbH \llra{B} \bbG \to 1$$
is {\bf short exact}  if in the
diagram
         $$\xymatrix@C=8pt@R=6pt@M=6pt{
               K_1 \ar[dd] \ar[rd] & & H_1  \ar[ld] \ar[rd] \ar '[d][dd]
                                         & & G_1 \ar[ld] \ar[dd] \\
                     &   C  \ar[rd] \ar @{..>} [rr]^(0.35){\delta}
                                 \ar[ld] & & B \ar[ld] \ar[rd] & \\
                                K_0 &&       H_0 & & G_0       }$$
we can find an arrow $\delta \: C \to B$ such that the diagram  is
commutative and the sequence
  $$1\to K_1 \llra{} C \llra{\delta} B \llra{} G_0 \to 1$$
is exact. (Note that $\delta$ is not necessarily unique.)

\begin{ex}{\label{E:strictexact}}
    Assume $C$ and $B$ are strict butterflies, that is, they come
    from strict morphisms $(c_1,c_0) \: [K_1\to K_0] \to [H_1 \to
    H_0]$ and $(b_1,b_0) \: [H_1\to H_0] \to [G_1 \to G_0]$ of
    crossed-modules (\cite{Maps}, $\S$9.5). Then, the sequence
        $$\xymatrix@C=8pt@R=6pt@M=6pt{
             1 \ar[dd] \ar[rr] &&  K_1   \ar[dd]^{\partial_{\bbK}}
                        \ar[rr]^{c_1}  & &
               H_1    \ar[rr]^{b_1} \ar[dd]^{\partial_{\bbH}}  & &
               G_1  \ar[dd]^{\partial_{\bbG}}   \ar[rr]  && 1\ar[dd] \\
                                      &&&&    & & & & \\
            1 \ar[rr]  && \ar[rr]  K_0 \ar[rr]_{c_0}     &&
                                H_0 \ar[rr]_{b_0} & & G_0
                       \ar[rr]    && 1   }$$
    is exact if and only if there exists a map $\psi \: K_0 \to G_1$
    such that:
     \begin{itemize}
        \item for every $k, k' \in K_0$, $\psi(kk')=\psi(k)^{b_0c_0(k')}\psi(k')$,
        \item the images of $\partial_{\bbG}$ and $b_0$ generate $G_0$,
        \item the intersection of the kernels of $\partial_{\bbK}$ and $c_1$ is trivial,
        \item for every $k \in K_0$, $b_0c_0(k)\cdot\partial\psi(k)=1$,
        \item for every $\gamma \in K_1$, $b_1c_1(\gamma)\cdot\psi(\partial\gamma)=1$,
        \item if $k \in K_0$ and $\be \in H_1$ are such that $c_0(k)\partial\be=1$ and
           $\psi(k)=b_1(\be)$, then there exists $\gamma \in K_1$ such that
           $k=\partial\gamma$ and $\be=c_1(\gamma)^{-1}$,
        \item if $h \in H_0$ and $\al \in G_1$ are such that $b_0(h)\partial\al=1$,
           then there exist $k\in K_0$ and $\be \in H_1$ such that
           $h=c_0(k)\partial\be$ and $\al=b_1(\be)^{-1}\psi(k)$.
     \end{itemize}
  Observe that the above list of conditions is equivalent to the sequence
      $$ 1\to K_1 \to K_0\ltimes H_1 \to H_0 \ltimes G_1 \to G_0 \to 1$$
  being exact.  The maps in this sequence are as follows:
      $$K_1 \to K_0\ltimes H_1,
          \ \gamma \mapsto (\partial\gamma\, , \, c_1\big(\gamma^{-1})\big);$$
      $$K_0\ltimes H_1 \to H_0 \ltimes G_1, \ \ (k,\beta) \mapsto
          \big(c_0(k)\partial\beta\, , \, b_1(\beta)^{-1} \psi(k)\big);$$
      $$H_0 \ltimes G_1 \to G_0, \ (h,\alpha) \mapsto  b_0(h)\partial(\alpha).$$
\end{ex}

The proof of the following proposition  will appear in \cite{Extensions}.

\begin{prop}{\label{P:shortexact}}
   A sequence of crossed-modules and butterflies
           $$1 \to \bbK \llra{C} \bbH \llra{B} \bbG \to 1$$
   is exact if and only if the induced sequence
            $$1 \to \mcK \llra{C} \mcH \llra{B} \mcG \to 1$$
   of group  stacks  is exact in the sense of (\cite{ButterflyI}, $\S$6.2).
\end{prop}

\subsection{Cohomology long exact sequence}{\label{SS:long}}

By applying (\cite{ButterflyI}, Proposition 6.4.1) to the site
$B\Gamma$ and making use of Proposition \ref{P:biequiv}, we
immediately obtain the following (also see \cite{CeFe}, Theorem 31).

\begin{prop}{\label{P:longexact}}
  Let $$1 \to \bbK \llra{C} \bbH \llra{B} \bbG \to 1$$ be a short
  exact sequence of $\Gamma$-crossed-modules and
  $\Gamma$-butterflies. Then, we have a long exact cohomology sequence
  {\small $$ \xymatrix@C=12pt@R=17pt{%
      1 \ar[r] & H^{-1} (\Gamma,\bbK) \ar[r] & H^{-1} (\Gamma,\bbH) \ar[r] &
      H^{-1} (\Gamma,\bbG) \ar[r] &
      H^{0} (\Gamma,\bbK) \ar `/8pt[d] `[lll] `[dlll] [dlll] \\
      & H^{0} (\Gamma,\bbH) \ar[r] & H^{0} (\Gamma,\bbG) \ar[r] &
      H^{1} (\Gamma,\bbK)
      \ar[r] & H^{1} (\Gamma,\bbH) \ar[r] & H^{1} (\Gamma,\bbG).
                }   $$  }
    (Remark that the connecting homomorphisms in  this long exact
    sequence depend on the choice of the homomorphism
    $\delta$ appearing in the definition of a short exact sequence.)
\end{prop}

The above proposition can be strengthened as follows (see \cite{CeFe},
Proposition 30).

\begin{prop}{\label{P:fibration}}
  Let
  $$1 \to \bbK \llra{C} \bbH \llra{B} \bbG \to 1$$
  be a short exact sequence of $\Gamma$-crossed-modules and
  $\Gamma$-butterflies. Then, the sequence
      $$\mfZ'(\Gamma,\bbK) \llra{C_*} \mfZ'(\Gamma,\bbH)
                \llra{B_*} \mfZ'(\Gamma,\bbG)$$
 is a fibration of 2-groupoids. The long exact sequence of
 Proposition \ref{P:longexact} is the fiber homotopy exact
 sequence associated to this fibration.
\end{prop}

The proof of the above proposition is not hard (and one can say it
is ``standard''), but it is
not in the spirit of these notes, so we omit it.

The following is an immediate corollary of Proposition \ref{P:longexact}.

\begin{prop}{\label{P:short}}
  Let $\bbG \: [\partial \: G_1 \to G_0]$ be  a $\Gamma$-crossed-module. Then,
  we have the exact sequences
    {\small $$ \xymatrix@C=6pt@R=11pt@M=4pt{%
      1 \ar[rr] && H^1(\Gamma,\ker\partial) \ar[rr]
                && H^0(\Gamma,\bbG)  \ar[rr]
                &&   (\coker\partial)^{\Gamma}
                    \ar `/7pt[d] `[l] `[dlll] `[r] [dll] && \\
          &&& & H^2(\Gamma,\ker\partial) \ar[rr]
          && H^1(\Gamma,\bbG) \ar[rr] &&
            H^1(\Gamma, \coker\partial)  }   $$  }
 and
   {\small $$ \xymatrix@C=6pt@R=11pt@M=4pt{%
      1 \ar[rr] && H^{-1}(\Gamma,\bbG) \ar[rr] && G_1 \ar[rr]
      && G_0 \ar[rr] && H^{0}(\Gamma,\bbG)  \ar `/7pt[d]
           `[l] `[dlllll] `[r] [dllll] && \\ && &&
            H^{1}(\Gamma,G_1) \ar[rr] &&
         H^{1}(\Gamma,G_0) \ar[rr] && H^{1}(\Gamma,\bbG).}$$}
\end{prop}

\begin{proof}
  For the first sequence, apply Proposition \ref{P:longexact} to
  the short exact  sequence of crossed-modules
    $$1 \to [\ker\partial \to 1] \to \bbG \to [1\to \coker\partial]\to 1.$$
  For the second sequence, apply Proposition \ref{P:longexact} to
  the short exact sequence
    $$ 1 \to G_1 \to G_0 \to \bbG \to 1.$$
  (To see why these two sequences of crossed-modules are exact use Example
  \ref{E:strictexact}.)
\end{proof}

\begin{rem}
The first exact sequence in Proposition \ref{P:short} can be
extended by adding an  $H^3(\Gamma,\ker\partial) $ to the right end
of it. We do not have the tools to give a systematic proof here but
with some effort one can prove it by hand. Also, in the second exact
sequence, if $G_1$ is abelian, the sequence can be extended by
$H^2(\Gamma,G_1)$. If $G_0$ is also abelian, then
the sequence can be extended further by $H^2(\Gamma,G_0)$.
\end{rem}

\begin{rem}
   The inclusion map $\ker\partial \to G_1$ and the projection map
   $G_0 \to \coker\partial$ induce maps $H^i(\Gamma, \ker\partial)
   \to H^i(\Gamma,G_1)$ and $H^i(\Gamma,G_0) \to H^i(\Gamma,
   \coker\partial)$. These maps intertwine the two
   exact sequences of Proposition \ref{P:short}  into a commutative diagram.
\end{rem}

\section{Appendix A: homotopy theoretic interpretation}{\label{A:0}}

In the previous sections we discussed three approaches to group cohomology with coefficients
in a $\Gamma$-equivariant 
crossed-module $\bbG$. 
In each approach we constructed a pointed homotopy 2-type (namely, the 
crossed-module in groupoids 
$\mcK^{\leq 1}(\Gamma,\bbG)$, and the 
2-groupoids $\mfZ(\Gamma,\bbG)$ and $\mfZ'(\Gamma,\bbG)$) whose homotopies give the 
desired group cohomologies. We showed that when $\bbG$ is braided (i.e., deloopable), then
the associated pointed homotopy 2-types are deloopable. If $\bbG$ is symmetric 
(i.e., double, hence 
infinitely, deloopable) then the associated pointed homotopy 2-types are double (hence,
infinitely) deloopable. 

We exhibited natural equivalences between these three pointed homotopy 2-types that
respect deloopings, and constructed cohomology exact sequences that  look like
(and indeed are, as we see below) the homotopy fiber exact sequence of a fibration.

There is a simple conceptual reasoning behind all this  that we would like 
to discuss in this appendix.  In what follows, we will be working in an  
`$(\infty,1)$-category $\mathbf{\infty Grpd}_S$ of $\infty$-groupoids
over a base $S$'. There are several ways to make sense of the phrase in quotes, 
each giving rise to a different approach to our group cohomology problem (we have seen three
so far). 
To mention a few more examples, we can take $S$ to be a topological space (say $B\Gamma$, the 
classifying space of $\Gamma$), and $\mathbf{\infty Grpd}_S$ the 
$(\infty,1)$-category of topological spaces over $S$.
Or we can take $S$ to be a category (say, the one-object category $B\Gamma$ 
with  morphisms  $\Gamma$) and $\mathbf{\infty Grpd}_S$ the 
$(\infty,1)$-category of $\infty$-groupoids fibered over 
$S$. Or we can take $S$ to be
a Grothendieck site (say, the site of the quotient stack $B\Gamma=[*/\Gamma]$) and 
$\mathbf{\infty Grpd}_S$ the $(\infty,1)$-category of simplicial (pre)sheaves 
(or $\infty$-stacks) over $S$.

\subsection{Definition of cohomology} 
 
Let us drop $S$ from the notation and denote
$\mathbf{\infty Grpd}_S$ by $\mathbf{\infty Grpd}$.
Let $X$ and $C$ be  objects in  $\mathbf{\infty Grpd}$.\footnote{Since we can always 
enlarge the base $S$ through a base change, there is no loss of generality in
assuming that $X$ is equal to the base $S$, i.e., it is a ``point''. But we will not do that here.}  
Assume that $C$ is pointed.
We think of $C$ as the `coefficients' of
our cohomology theory. We define
  $$H^0(X,C):=\pi_0\mathbf{\infty Grpd}(X,C),$$
where  $\mathbf{\infty Grpd}(X,C)$ is the pointed 
$\infty$-groupoid of morphisms from $X$ to $C$.
If  $\Omega \: \mathbf{\infty Grpd}_* \to \mathbf{\infty Grpd}_*$ is a loop
functor (so, $\Omega(Y)$  is homotopy equivalent to the homotopy fiber
product over $Y$ of the base point of $Y$ with itself), we  define
  $$H^{-n}(X,C):=\pi_0\mathbf{\infty Grpd}(X,\Omega^nC)=H^0(X,\Omega^nC), \ n\geq 0.$$
To define $H^n$ for positive $n$ we need to deloop $C$. If we choose an $n$-delooping of
$C$ and denote it by $\operatorname{B}^n\!C$, we can then define
   $$H^{n}(X,C):=\pi_0\mathbf{\infty Grpd}(X,\operatorname{B}^n\!C)
   =H^0(X,\operatorname{B}^n\!C), \ n\geq 0.$$
(This explains why to define all $H^n$ the coefficients $C$ should be taken to
be a spectrum, i.e, we need to fix an infinite sequence of iterated deloopings for $C$.)  

\begin{ex}{\label{E:example}}
Two examples to keep in mind are the following. 

\begin{itemize}

\item[1)]
When $A$ is a discrete abelian group (and the base $S$ is a point), 
$\operatorname{B}^n\!A$ is the $n$-th 
Eilenberg-MacLane object $K(n,A)$  which can be realized as the
one-object $\infty$-groupoid whose $n$-morphisms are $A$ and whose $k$-morphisms, 
$k\neq n$,
are only the identities. In the relative case over $S$, this construction needs to be adjusted
according to the base $S$.
For example, if $S$ is a Grothendieck site, then we take a sheaf $A_S$
of abelian groups over $S$ and use a sheafified version of the
above construction. In the case where $S=B\Gamma$ is the one-object category with morphisms
$\Gamma$, and $A$ is a $\Gamma$-equivariant abelian group, we need to work with
the action groupoid $A_{\Gamma}$ of the action of $\Gamma$ on $A$. (We view $A_{\Gamma}$  
as a groupoid fibered over $B\Gamma$). The delooping
 $\operatorname{B}^n\!A_{\Gamma}$, $n\geq 1$, is
the Eilenberg-MacLane object $K(n,A)\rtimes\Gamma$ of $A_{\Gamma}$ which is,
by definition, generated from $K(n,A)$
by adding 1-morphisms coming from $\Gamma$.

In this situation, 
the cohomologies defined above correspond to the usual cohomologies $H^n(X,A_S)$
(sheaf cohomology in the first case, and group cohomology in the second).

\item[2)] 
If $\bbG=[G_1 \to G_0]$ is a crossed-module and $\mcG$  its associated groupoid,
then $\mcG$ can always be delooped once. We denote its delooping  by 
$\operatorname{B}\!\mcG$. It is the 2-group associated to $\bbG$, that is,
the one-object 2-groupoid with morphisms $G_0$ and 2-morphisms $G_1\ltimes G_0$. 
If $\bbG$ is braided, then we can deloop once more to get $\operatorname{B}^2\!\mcG$. 
If $\bbG$ is symmetric, then $\mcG$ can be
delooped infinitely many times. 

The same thing is true in the relative case. For example,
if $S=B\Gamma$ is  the one-object category 
with morphisms $\Gamma$, and $\bbG$ is a $\Gamma$-equivariant crossed-module, 
then $\mcG$ should be replaced by the translation groupoid $\mcG_{\Gamma}$ of the
action of $\Gamma$ on $\mcG$.  The groupoid $\mcG_{\Gamma}$
is fibered over $B\Gamma$ and
is always  deloopable. Its delooping $B\mcG_{\Gamma}$
is the 2-group associated to the crossed-module $\bbG\rtimes \Gamma$.

It follows that, in general, for a crossed-module $\bbG_S$ over a base $S$, we can define 
  $$H^n(X,\bbG_S):= H^{n-1}(X,\operatorname{B}\!\mcG_S), \ n\leq 1.$$ 
(It turns out that  $H^n(X,\bbG_S)$ is trivial for 
$n\leq -2$ as $\Omega^2\mcG_{\Gamma}$ is contractible.) 
If $\bbG_S$ is braided, then we can also define $H^2(X,\bbG_S)$.
In the symmetric case, it is possible to define all $H^n(X,\bbG_S)$.

\end{itemize}

\end{ex}

Since the internal hom functor $\mathbf{\infty Grpd}(X,-)$ preserves
homotopy limits, hence in particular the homotopy fiber product of the base point with itself
that defines loop objects, it follows that the pointed $\infty$-groupoid 
$\mathbf{\infty Grpd}(X,C)$ packages the cohomology groups in the following way:
  $$H^{i}(X,C)\cong \pi_{-i}\mathbf{\infty Grpd}(X,C), \ i\leq 0.$$
If $\operatorname{B}^n\!C$ is an $n$-delooping of $C$, we have
     $$H^{i}(X,C)\cong \pi_{n-i}\mathbf{\infty Grpd}(X,\operatorname{B}^n\!C), i\leq n.$$
If we apply this to Example \ref{E:example} (2) above, we find that  
$\mathbf{\infty Grpd_{B\Gamma}}(B\Gamma,\operatorname{B}\!\mcG_{\Gamma})$ 
is a pointed 2-type
that calculates group cohomology with coefficients in $\bbG$, that is,
 $$H^i(\Gamma,\bbG)=\pi_{1-i}
   \mathbf{\infty Grpd_{B\Gamma}}(B\Gamma,\operatorname{B}\!\mcG_{\Gamma}), \ i=-1,0,1.$$   

The 
pointed crossed-module in groupoids 
$\mcK^{\leq 1}(\Gamma,\bbG)$, and the pointed
2-groupoids $\mfZ(\Gamma,\bbG)$ and $\mfZ'(\Gamma,\bbG)$ are all models for
the pointed 2-type
$\mathbf{\infty Grpd_{B\Gamma}}(B\Gamma,\operatorname{B}\!\mcG_{\Gamma})$.

\subsection{Cohomology long exact sequence} 

Let $A$, $B$ and $C$ be pointed objects in $\mathbf{\infty Grpd_{S}}$ viewed as 
coefficients. Assume that they fit in a (homotopy) fibration sequence
 $$A \to B \to C.$$
This sequence can  be extended to a sequence
 $$\cdots \to \Omega^2C \to \Omega A \to \Omega B \to \Omega C \to
  A \to B \to C$$
in which each term is the homotopy fiber of the morphism between the next two terms.
Applying the functor $\pi_0\mathbf{\infty Grpd}(X,-)$, we get an exact
sequence of pointed cohomology sets
     {\small $$ \xymatrix@C=12pt@R=17pt{%
      \cdots \ar[r] &  H^{-2}(X,A) \ar[r] &  H^{-2}(X,B) \ar[r] &
      H^{-2}(X,C)\ar[r] &
      H^{-1}(X,A) \ar `/8pt[d] `[lll] `[dlll] [dlll] \\
      & H^{-1}(X,B) \ar[r] & H^{-1}(X,C) \ar[r] &
       H^{0}(X,A)
      \ar[r] &  H^{0}(X,B) \ar[r] &  H^{0}(X,C).
                }   $$  }
               
If we are in the situation of Example \ref{E:example} (2), and 
    $$1 \to \bbK_S \to \bbH_S \to \bbG_S \to 1$$
is a short exact sequence of crossed-modules over $S$, then   
  $$\operatorname{B}\mathcal{K}_S \to \operatorname{B}\mcH_S \to 
    \operatorname{B}\mcG_S$$
is a fibration sequence in $\mathbf{\infty Grpd_{S}}$ and the above
cohomology exact sequence coincides
with the one of Proposition \ref{P:longexact}.
             
\section{Appendix B: review of 2-crossed-modules and braided
crossed-modules}{\label{A:1}}

For the convenience of the reader, in this appendix we collect some
elementary facts about braided crossed-modules and 2-crossed-modules
\cite{Brown-Gilbert, Conduche}. We begin with some definitions.

A {\bf crossed-module in groupoids} (\cite{Brown-Gilbert}, page 54) is a
morphism of groupoids
   $$\mathcal{M} \llra{\partial} \mathcal{N}$$
such that  $\mathcal{M}=\coprod_{x \in
\Ob(\mathcal{N})}\mathcal{M}(x)$ is a disjoint union of groups
indexed by the set of objects of $\mathcal{N}$. We also have a right
action of  $\mathcal{N}$ on $\mathcal{M}$ such that an arrow $g \in
\mathcal{N}(x,y)$ takes $\al \in \mathcal{M}(x)$ to $\al^g \in
\mathcal{M}(y)$. We require that $\partial$ satisfy the two axioms
of a crossed-module. That is, $\partial$ is
$\mathcal{N}$-equivariant for the right conjugation action of
$\mathcal{N}$ on itself, and for every two arrows $\al, \be$ in
$\mathcal{M}$, we have $\al^{\partial\be}=\be^{-1}\al\be$.

Any crossed-module $[M \to N]$ gives rise to  a crossed-module
in groupoids
     $$\big[M \to [N \sst{} 1]\big].$$
Conversely, to any object $x$ in crossed-module in groupoids
$[\partial \: \mathcal{M} \to \mathcal{N}]$  we can associate a
crossed-module $[\partial_x \: \mathcal{M}(x)  \to \mathcal{N}(x)]$
which we call the {\em automorphism crossed-module} of $x$. Here, by
$\mathcal{N}(x)$ we mean the automorphism group of the object $x \in
\mathcal{N}$.

A {\bf $2$-crossed-module} (\cite{Conduche}, Definition 2.2; also
see \cite{Brown-Gilbert}, page 66), is a sequence
  $$[L \llra{\partial} M \llra{\partial} N]$$
of groups endowed with a right action action of $N$ on $M$ and $L$,
a right action of $M$ on $L$, and a bracket $\{,\} \: M\times M \to
N$ satisfying the following axioms:
   \begin{itemize}
      \item Let $N$ act on itself by right conjugation. Then both
         differentials $\partial$ are $G_1$-equivariant, and $\partial^2=0$;
      \item  For every $g,h \in M$, $\partial\{g,h\}=g^{-1}h^{-1}gh^{\partial g}$;
      \item  For every $g \in M$ and $\al \in L$, $\{\partial\al,g\}=\al^{-1}\al^g$ and
        $\{g,\partial\al\}=(\al^{-1})^g\al^{\partial g}$;
      \item For every $g,h,k \in M$, $\{g,hk\}=\{g,k\}\{g,h\}^{k^{\partial g}}$;
      \item For every $g,h,k \in M$, $\{gh,k\}=\{g,k\}^h\{h,k^{\partial g}\}$;
      \item or every $g,h\in M$ and $x \in N$, $\{g,h\}^x=\{g^x,h^x\}$.
   \end{itemize}

By setting $N=\{1\}$ in the definition of a 2-crossed-module, we
obtain the definition of a braided crossed-modules. More precisely,
a crossed-module
           $$[L \to M]$$
is {\bf braided} if it is endowed with a bracket
$\{,\} \: M\times M \to L$ which satisfies the
following axioms:
     \begin{itemize}
      \item  For every $g,h \in M$, $\partial\{g,h\}=g^{-1}h^{-1}gh$;
      \item  For every $g \in M$ and $\al \in L$,
        $\{\partial\al,g\}=\al^{-1}\al^g$ and
        $\{g,\partial\al\}=(\al^{-1})^g\al$;
      \item For every $g,h,k \in M$, $\{g,hk\}=\{g,k\}\{g,h\}^{k}$;
      \item For every $g,h,k \in M$, $\{gh,k\}=\{g,k\}^h\{h,k\}$.
   \end{itemize}

Any 2-crossed-module $[L\to M \to N]$ gives rise to a crossed-module
in groupoids
    $$\big[\coprod_{x \in N}L(x) \to [N\times M \sst{} N]\big],$$
where $L(x)=L$ and  $[N\times M \sst{} N]$ is the action groupoid of
the right multiplication action of $M$ on $N$ via $\partial$. If we
view $1 \in N$  as an object in the above crossed-module in
groupoids, its automorphism crossed-module is equal to $[L \to \ker
\partial]$. This is a braided crossed-module. Conversely, any
braided crossed-module $[L \to M]$
gives rise to a 2-crossed-module
    $$[L \llra{\partial} M \llra{\partial} 1].$$

A braided crossed-module $[L \to M]$ is {\bf symmetric} if for every
$g, h \in M$ we have
   $$\{g,h\}\{h,g\}=1.$$
If, in addition, we have
   $$\{g,g\}=1$$
for every $g \in M$, we say that $[L \to M]$ is {\bf Picard}.

The above observation about braided crossed-modules can be used 
to define a {\bf braided  2-crossed-module} as follows.\footnote{To our knowledege,
braided 2-crossed-modules were first defined in P.~Carrasco's thesis \oldcite{Carrasco}.} 
We say that a 
2-crossed-module 
   $$[K \llra{\partial} L \llra{\partial} M]$$
is   braided if the sequence
  $$K \llra{\partial} L \llra{\partial} M \llra{\partial} 1$$
is endowed with the structure of a 3-crossed-module in the sense of \cite{ArKuUs}, 
Definition 8.
That is, we have seven brackets
  $$\begin{array}{rl}
      \{,\}_{(1)(0)}, \ \ \{,\}_{(0)(2)},\ \ \{,\}_{(2)(1)}  : & L\times L \to K \\
      \{,\}_{(1,0)(2)}, \ \ \{,\}_{(2,0)(1)}  : & M\times L \to K \\
      \{,\}_{(0)(2,1)}  : &L\times M \to K \\ 
       \{,\}  : & M\times M \to L
   \end{array}$$
satisfying axioms ($\mathbf{3CM1}$)-($\mathbf{3CM18}$) 
of {\em loc.~cit.}\footnote{We need to modify the axioms of {\em loc.~cit} to account for the 
fact that our conventions for the actions (left or right) and the 
brackets, and as a consequenece our 
2-crossed-module axioms, are different from those of {\em loc.~cit}.} In fact, 
it follows from the axioms that 
the two brackets $\{,\}_{(1)(0)}$ and $\{,\}_{(0)(2)}$ are determined by $\{,\}_{(2)(1)}$, 
and $\{,\}_{(2)(1)}$ itself is the bracket that already 
comes with the 2-crossed-module. So, to put a braiding on a given 2-crossed-module we
have to introduce four new brackets, namely,  the last four in the above list.   
(In our application in $\S$\ref{SS:symmetric}, three of these four
brackets are trivial and only $\{,\}  :  M\times M \to L$ is non-trivial.)

\begin{rem}
  It is  useful to keep in mind the homotopy theoretic
  interpretations of the above notions. A crossed-module corresponds
  to a pointed homotopy 2-type. A crossed-module in groupoids
  corresponds to an arbitrary homotopy 2-type. A 2-crossed-module
  corresponds to a pointed homotopy 3-type. Associating a
  crossed-module in groupoids to a 2-crossed-module corresponds to
  taking the based loop space. The 2-crossed-module associated to a
  braided crossed-module corresponds to delooping.
\end{rem}

\subsection{Cohomologies of a crossed-module in groupoids}{\label{SS:cohs}}

To be compatible with the rest of the paper, we make the (unusual)
assumption that our crossed-module in groupoids $[\partial
\:\mathcal{M} \to \mathcal{N}]$ is sitting in degrees $[-1,1]$. That
is, we think of objects of $\mathcal{N}$ as sitting in degree $1$,
its arrows in degree $0$, and arrows of $\mathcal{M}$ in degree
$-1$. We then define $H^{1}$ to be the set of connected components
of $\mathcal{N}$; this is just a set. For a  fixed a base point $x
\in \Ob(\mathcal{N})$, we define $H^0$ and $H^{-1}$ to be,
respectively, the cokernel and the kernel of $\partial_x$ in the
automorphism crossed-module $[\partial_x \:\mathcal{M}(x) \to
\mathcal{N}(x)]$ of $x$. Note that $H^0$ is a group and $H^{-1}$ is
an abelian group

In the case where our crossed-module in groupoids comes from a
2-crossed-module $L \to M\to N$,  concentrated in degrees $[-1,1]$,
the cohomologies defined above are naturally isomorphic to the
cohomologies of the 2-crossed-module. In this situation, $H^1$ is
also a group and $H^0$ is abelian.

\subsection{The 2-groupoid associated to a crossed-module in
groupoids}{\label{SS;2gpd}}

To any crossed-module in groupoids $[\partial \: \mathcal{M} \to
\mathcal{N}]$ we can associate a strict 2-groupoid
$[\mathcal{N}/\mathcal{M}]$ as follows. The objects and the arrows
of $[\mathcal{N}/\mathcal{M}]$ are the ones of $\mathcal{N}$. Given
two arrows $g,h \in \mathcal{N}(x,y)$, a 2-arrow  $g \Rightarrow h$
is an element $\al \in \mathcal{M}(y)$ such that
$g\partial_y(\al)=h$. The composition of two 2-arrows $\al \: g
\Rightarrow h$ and $\be \: h \Rightarrow k$ is $\al\be$. If $k \in
\mathcal{N}(y,z)$, then $\al k \: gk \Rightarrow hk$ is defined to
be the 2-arrow corresponding to $\al^k \in \mathcal{N}(z)$. If $k
\in \mathcal{N}(z,x)$, then $k\al  \: kg \Rightarrow kh$ is defined
to be the 2-arrow corresponding to $\al$ itself.

In the case where  $[\partial \: \mathcal{M} \to \mathcal{N}]$ comes
from a 2-crossed-module, the 2-groupoid $[\mathcal{N}/\mathcal{M}]$
can be delooped to a 3-group. That is, there is a (weak) 3-groupoid
with one object such that  the morphisms from the unique object to
itself is equal to $[\mathcal{N}/\mathcal{M}]$. This is true because
$[\mathcal{N}/\mathcal{M}]$ is a strict group object in the category
of 2-groupoids. (Note that, although this is a strict group object,
the multiplication functor is lax. The laxness of the multiplication
functor is measured by the bracket of the  2-crossed-module.)

\subsection{Some useful identities}{\label{SS:identities}}

The following identities are frequently used in the (omitted) proofs
of many of the claims in these notes. The proofs are left to the
reader. In what follows $[M \to N]$ is a braided crossed-module.

\begin{itemize}
   \item For every $g,h \in N$, $\{g,h^{-1}\}^h=\{g,h\}^{-1}
          =\{g^{-1},h\}^g$.
   \item For every $g,h \in N$,  $\{g^{-1},h^{-1}\}^{gh}=\{g,h\}$.
   \item For every $g \in  N$, $\{g,g\}^g=\{g,g\}$.
   \item For every $g,h,k \in N$, $\{gh,k\}=\{h,g^{-1}kg\}\{g,k\}$.
   \item For every $g,h,k \in N$. $\{g,hk\}=\{g,h\}\{h^{-1}gh,k\}$.
   \item For every $g,h,k \in N$, $\{g,h\}^k=\{k^{-1}gk,k^{-1}hk\}$.

\end{itemize}

\providecommand{\bysame}{\leavevmode\hbox
to3em{\hrulefill}\thinspace}
\providecommand{\MR}{\relax\ifhmode\unskip\space\fi MR }
\providecommand{\MRhref}[2]{%
  \href{http://www.ams.org/mathscinet-getitem?mr=#1}{#2}
} \providecommand{\href}[2]{#2}

\end{document}